\documentclass{amsart}
\usepackage{amsmath}
\usepackage{amsfonts}
\usepackage{amstext}
\usepackage{amsbsy}
\usepackage{amsopn}
\usepackage{amsxtra}
\usepackage{upref}
\usepackage{amsthm}
\usepackage{amsmath}
\usepackage{amssymb}
\DeclareMathOperator{\var}{var}
\newtheorem{prop}{Proposition}[section]
\newtheorem{rem}{Remark}[section]
\newtheorem{lema}{Lemma}[section]
\newtheorem{defi}{Definition}[section]
\newtheorem{teo}{Theorem}[section]                                             
\newtheorem{eje}{Example}[section]
\newtheorem{coro}{Corollary}[section]

\def\R{{\mathbb R}}

\def\N{{\mathbb N}}

\title[Multifractal analysis of Birkhoff averages]{Multifractal analysis of Birkhoff averages for countable Markov maps}
\date{\today}

\begin{thanks}
{GI was partially  supported by Proyecto Fondecyt 11070050 and Proyecto Fondecyt 1110040. TJ wishes to thank the Chilean government for funding his visit to Chile. We'd also like to thank Henry Reeve for his careful reading of an earlier version of the paper which led to several improvements.}
\end{thanks}

\author{Godofredo Iommi} \address{Facultad de Matem\'aticas,
Pontificia Universidad Cat\'olica de Chile (PUC), Avenida Vicu\~na Mackenna 4860, Santiago, Chile}
\email{giommi@mat.puc.cl}
\urladdr{http://www.mat.puc.cl/\textasciitilde giommi/}
\author{Thomas Jordan} \address{The School of Mathematics, The University of Bristol, University Walk, Clifton, Bristol, BS8 1TW, UK}
\email{Thomas.Jordan@bristol.ac.uk}
\urladdr{http://www.maths.bris.ac.uk/~matmj}

\begin{document}

\begin{abstract}
In this paper we prove a multifractal formalism of Birkhoff averages for interval maps with countably many branches. Furthermore, we prove that under certain  assumptions  the Birkhoff spectrum is real analytic. We also show that new phenomena occurs, indeed  the spectrum can be constant or it can have points where it is not analytic. Conditions for these to happen are obtained.  Applications of these results to number theory are also given. Finally, we compute the Hausdorff dimension of the set of points for which the Birkhoff average is infinite.
\end{abstract}

\maketitle

\section{Introduction} 
The Birkhoff average of a regular function with respect to an hyperbolic dynamical system can take a wide range of values.  This paper is devoted to  study the fine structure of level sets determined by Birkhoff averages.  The class of dynamical systems we consider are interval maps with countably many branches. These maps can be modeled by the (non-compact) full-shift on a countable alphabet. The lack of compactness of this model, and the associated convergence problems, is one of the major difficulties that has to be overcome in order to obtain a precise description of the level sets.

Let us be more precise, denote by $I=[0,1]$  the unit interval. We consider the  class of EMR (expanding-Markov-Renyi) interval maps. This class was considered  by Pollicott and Weiss in \cite{pw} when studying multifractal analysis of pointwise dimension. 
\begin{defi} \label{maps}
A map $T:I \to I$ is an EMR map, if there exists a countable family $\{ I_i \}_{i}$ of closed intervals (with disjoint interiors $\textrm{int }I_n$) with $I_i \subset I$ for every $i \in \mathbb{N}$, satisfying
\begin{enumerate}
\item The map is $C^2$ on $\cup_{i=1}^{\infty} \textrm{int }I_i$.
\item There exists $\xi >1$ and $N\in\N$ such that for every $x \in \cup_{i=1}^{\infty} I_i$ and $n\geq N$
we have $|(T^n)'(x)|>\xi^n$.
\item The map $T$ is Markov and it can be coded by a full-shift on a countable alphabet. 
\item The map satisfies the Renyi condition, that is, there exists a positive number $K>0$
such that
\[ \sup_{n \in \N} \sup_{x,y,z \in I_n} \frac{|T''(x)|}{|T'(y)| |T'(z)|} \leq K. \]                  
\end{enumerate}
\end{defi}

The \emph{repeller} of such a map is defined by
\[\Lambda:=\left\{x \in \cup_{i=1}^{\infty} I_i: T^n(x) \textrm{ is well defined for every } n \in \mathbb{N} \right\}.\]
For simplicity  we will also assume that zero is the unique  accumulation point  of the set of endpoints of $\{I_i\}$.

\begin{eje} \label{ga}
The Gauss map $G:(0,1] \to (0,1]$ defined by
\[G(x)= \frac{1}{x} -\left[ \frac{1}{x} \right],\]
where $[ \cdot]$ is the integer part, is an EMR map.
\end{eje}
The ergodic theory of EMR maps can be studied using its symbolic model and the available results for countable Markov shifts. We follow this strategy in order to describe the thermodynamic formalism for EMR maps  for a large class of potentials (see Section \ref{st}).

Let $\phi: \Lambda  \to \mathbb{R}$ be a continuous function. We will be interested in the level sets determined by the Birkhoff averages of $\phi$. Let 
\begin{eqnarray*}
\alpha_m=\inf \left\{ \lim_{n \to \infty} \frac{1}{n} \sum_{i=0}^{n-1} \phi (T^i x): x \in  \Lambda \right\} \textrm{ and }&\\
\alpha_M=\sup \left\{ \lim_{n \to \infty} \frac{1}{n} \sum_{i=0}^{n-1} \phi (T^i x): x \in  \Lambda \right\}.
\end{eqnarray*}
Note that, since the space $\Lambda$ is not compact,  it is possible for $\alpha_m$ and $\alpha_M$ to be minus infinity and infinity respectively.
For $\alpha \in [\alpha_m,\alpha_M]$ we define the level set  of points having Birkhoff average equal to $\alpha$ by
\begin{equation*}
J(\alpha):= \left\{x \in  \Lambda : \lim_{n \to \infty} \frac{1}{n} \sum_{i=0}^{n-1} \phi (T^i x) = \alpha \right\}. 
\end{equation*}
Note that these sets induce the so called \emph{multifractal decomposition} of the repeller,
\begin{equation*}
\Lambda= \bigcup_{\alpha=\alpha_m}^{\alpha_M}J(\alpha) \text{ } \bigcup J',
\end{equation*}
where $J'$ is the \emph{irregular set} defined by,
\[J' := \left\{x \in  \Lambda :  \textrm{ the limit }\lim_{n \to \infty} \frac{1}{n} \sum_{i=0}^{n-1} \phi (T^i x)  \textrm { does not exists }  \right\}. \]
The \emph{multifractal spectrum} is the function that encodes this decomposition and it is defined by
\[b(\alpha)= \dim_H(J(\alpha)),\]
where $\dim_H(\cdot)$ denotes the Hausdorff dimension (see Subsection \ref{hausdorff}).

The function $b(\alpha)$ has been studied in the context of hyperbolic dynamical systems  (for instance EMR maps with a finite Markov partition) for potentials with different degrees of regularity. Initially this was studied in the symbolic space for H\"{o}lder potentials by  Pesin and Weiss \cite{pw2} and for general continuous potentials  by Fan, Feng and Wu \cite{flw}.  Lao and Wu, \cite{flw}, then studied the case of continuous potentials for conformal expanding maps.  Barreira and Saussol \cite{bs1} showed that the multifractal spectrum for H\"{o}lder continuous functions is real analytic in the setting of conformal expanding maps. They stated their results in terms of variational formulas.  Olsen \cite{olsen}, in a similar setting  obtained more general variational formulae for families of continuous potentials. The multifractal analysis for Birkhoff averages for some non-uniformly hyperbolic maps (such as Manneville Pomeau) was studied by 
Johansson, Jordan, \"{O}berg and Pollicott in \cite{jjop}. There have also been several articles on multifractal analysis in the countable state case see for example \cite{fl, hmu, I1, KS}. However, these papers look at the local dimension spectra or the Birkhoff spectra for very specific potentials (e.g. the Lyapunov spectrum).

Our main result is that in the context of EMR maps we can make a variational characterisation of the multifractal spectrum,
\begin{teo}\label{main} 
Let $\phi \in \mathcal{R}$ be a potential then for $\alpha\in (-\infty,\alpha_M)$ we have that
\begin{equation}
b(\alpha) = \sup \left\{ \frac{h(\mu)}{\lambda(\mu)} :\mu \in \mathcal{M}_T, \int  \phi \, d \mu = \alpha \textrm{ and } \lambda(\mu) < \infty \right\},
\end{equation}
where the class $\mathcal{R}$ is defined in Subsection \ref{symbolic}, $\mathcal{M}_T$ denotes the set of $T-$invariant probability measures, $h(\mu)$ denotes the measure theoretic entropy and $\lambda(\mu)$ is the Lyapunov exponent (see Section \ref{st}).
\end{teo}

The other major result, which we proof in Section \ref{R}, is that when $\phi$ is sufficiently regular and satisfies certain asymptotic behaviour as $x\to 0$ the multifractal spectrum
has strong regularity properties.
\begin{teo} \label{analytic}
Let $\phi\in\bar{\mathcal{R}}$ be a potential. The following statements hold.
\begin{enumerate}
\item
If $\lim_{x\to 0}\frac{\phi(x)}{-\log |T'(x)|}=\infty$ and there exists an ergodic measure of full dimension $\mu$ then 
$b(\alpha)$ is real analytic on $\left(\int\phi\text{d}\mu,\alpha_M\right )$ and $b(\alpha)=\dim\Lambda$ for all $\alpha\leq\int\phi\text{d}\mu$.
\item
If $\lim_{x\to 0}\frac{\phi(x)}{-\log |T'(x)|}=\infty$ and there does not exist an ergodic measure of full dimension then $b(\alpha)$ is real analytic for all $\alpha\in (-\infty,\alpha_M)$. 
\item
If $\lim_{x\to 0}\frac{\phi(x)}{-\log |T'(x)|}=0$ then there are at most two point when $b(\alpha)$ is non-analytic. 
\end{enumerate}
\end{teo}
Without the assumptions made in Theorem \ref{analytic} it is hard to say anything in general but it is possible to say things in specific cases. 
 We investigate this further in Sections 5 and 6. In particular in Section \ref{sh}  we look at the case when $\phi(x)=-\log |T'|$ and we also look at the shapes $b(\alpha)$ can take.  
 
%  The functional $P(\cdot)$ denotes the pressure and it is defined in Section \ref{symbolic}.
In Section \ref{ejem} we apply the above two theorems to the Gauss map and obtain results relating to the continued fraction expansion.  Our results relate to classical ones by Khinchine \cite{k} regarding the 
size of sets determined by averaging values of the digits in the continued fraction expansion of irrational numbers. We not only consider the behaviour of the limit
\begin{equation*}
 \lim_{n \to \infty} \sqrt[n]{a_1 \cdot a_2 \cdot  \ldots  \cdot a_n} ,
\end{equation*}
where the continued fraction expansion of $x$ is given by $[a_1 a_2 \dots]$. But we generalise it to a wide range of other functions. For example, we are able to describe level sets determined by the arithmetic averages of the digits in the continued fraction:
\begin{equation*} 
\lim_{n \to \infty} \frac{1}{n} \left(a_1+a_2 +\dots +a_n   \right).
\end{equation*}
Note that there is related work in \cite{flm} where they look at the dimension of the sets where the frequencies of values the $a_i$ can take are prescribed.

Since the potentials we consider are unbounded their Birkhoff average can be infinite. In Section \ref{extr} we compute the Hausdorff dimension of the set of points for which the Birkhoff average is infinite.

\section{Symbolic model and Thermodynamic Formalism} \label{st}
In this Section we describe the thermodynamic formalism for EMR maps. In order to do so, we will first recall results describing the thermodynamic formalism in the symbolic setting.

\subsection{Thermodynamic formalism for countable Markov shifts} \label{CMS}
The full-shift on the countable alphabet $\mathbb{N}$ is the pair $(\Sigma, \sigma)$ where
\[
\Sigma = \left\{ (x_i)_{i \ge 1} : x_i \in \mathbb{N}\right\},
\]
and $\sigma: \Sigma \to \Sigma$ is the \emph{shift} map  defined by $\sigma(x_1x_2 \cdots)=(x_2 x_3\cdots)$.
We equip $\Sigma$ with the topology generated by the cylinders sets
\[ C_{i_1 \cdots i_n}= \{x \in \Sigma : x_j=i_j \text{ for } 1 \le j \le n \}.\] 
The $n-$variation of a function  $\phi: \Sigma \to \mathbb{R}$ are defined by
\[
V_{n}(\phi) := \sup \left\{|\phi(x)-\phi(y)| : x,y \in \Sigma,\
x_{i}=y_{i} \text{ for } 0 \le i \le n-1 \right\}.
\]
We say that a function $\phi: \Sigma \to \mathbb{R}$ has \emph{summable variation} if
$\sum_{n=2}^{\infty} V_n(\phi)<\infty$. If $\phi$ has summable variation then it is continuous.  A function $\phi: \Sigma \to \R$ is called \emph{weakly H\"older} if there exist $A > 0$ and $\theta \in (0,1)$ such that for all $n \geq 2$ we have $V_n(\phi) \leq A \theta^n$. The thermodynamic formalism is well understood for the full-shift on a countable alphabet. The following definition of pressure is due to Mauldin and Urba\'nski \cite{mu},

\begin{defi}
Let $\phi: \Sigma \to \mathbb{R}$ be a potential of summable variations, the \emph{pressure} of $\phi$ is defined by
\begin{equation}
P(\phi) = \lim_{n \to \infty} \frac{1}{n} \log \sum_{\sigma^n(x)=x} \exp \left( \sum_{i=0}^{n-1} \phi(\sigma^i x) \right).
\end{equation}
\end{defi}
The above limit always exits,  but it can be infinity. This notion of pressure satisfies the following results (see \cite{mu, sa1,sa2,sa3}),

\begin{prop}[Variational Principle]
If $\phi: \Sigma \to \mathbb{R}$ has summable variations and $P(\phi)<\infty$ then
\begin{equation*}
P(\phi)= \sup \left\{ h(\mu) +\int  \phi \, d
\mu : -\int  \phi \, d
\mu < \infty \textrm{ and } \mu\in \mathcal{M}_{\sigma}  \right\},
\end{equation*}
where $\mathcal{M}_{\sigma}$ is the space of shift invariant probability measures and $h(\mu)$ is the measure theoretic entropy (see \cite[Chapter 4]{wa}).
\end{prop}

\begin{defi}
Let $\phi: \Sigma \to \mathbb{R}$ be a potential of summable variations. A measure $\mu \in \mathcal{M}_{\sigma}$ is called an \emph{equilibrium measure} for $\phi$ if
\[ P(\phi)=  h(\mu) +\int  \phi \, d \mu .\]
  \end{defi}

\begin{prop}[Approximation property] \label{app}
If $\phi: \Sigma \to \mathbb{R}$ has summable variations  then
\begin{equation*}
P( \phi) = \sup \{ P_{\sigma|K}( \phi) : K \subset \Sigma : K \ne \emptyset \text{ compact and } \sigma\text{-invariant}  \},
\end{equation*}
where $P_{\sigma|K}( \phi)$ is the classical topological pressure on
$K$ (for a precise definition see \cite[Chapter 9]{wa}).
\end{prop}
\begin{defi} \label{Gibbs}
A probability measure $\mu$ is called a \emph{Gibbs measure} for the potential $ \phi$ if there exists two constants $M$ and $P$, such that for every cylinder $C_{i_{1} \dots i_{n}}$ and every $x \in C_{i_{1}, \dots i_{n}}$ we have that
\[ \frac{1}{M} \leq \frac{ \mu (C_{i_{1} \dots i_{n}})}{\exp(-nP + \sum_{j=0}^{n-1} \phi(\sigma^{j}x))} \leq M. \]
\end{defi}
\begin{prop}[Gibbs measures]
Let $\phi: \Sigma \to \mathbb{R}$ be a potential such that $\sum_{n=1}^{\infty} V_n(\phi)<\infty$ and $P(\phi) < \infty$ then $\phi$ has a unique Gibbs measure. 
\end{prop}

\begin{prop}[Regularity of the pressure function]
Let $\phi: \Sigma \to \mathbb{R}$ be a weakly H\"older potential such that $P(\phi) < \infty$ , there exists a \emph{critical value} $s^{*} \in (0, 1]$ such that for every $s <s^{*}$we have that $P(s \phi)= \infty$ and for every $s > s^{*}$we have that $P(s \phi)< \infty$. Moreover, if $s>s^{*}$ then
the function $s \to P(s \phi)$  is real analytic and every potential $s \phi$ has an unique equilibrium measure.
\end{prop}

\subsection{Symbolic model} \label{symbolic}
It is a direct consequence of the  Markov structure assumed on a EMR map $T$ that  
 $T:\Lambda \to \Lambda$ can be represented by a full-shift on a countable alphabet $(\Sigma, \sigma)$. Indeed, there exists a continuous map $\pi :\Sigma \to \Lambda$ such that $\pi \circ \sigma = T  \circ  \pi$.  Moreover, if we denote by $E$ the set of end points of the partition $\{I_i\}$, the map $\pi:\Sigma \to \Lambda \setminus \bigcup_{n \in \N} T^{-n} E$ is an homeomorphism.  Denote by $I(i_1, \dots i_n)= \pi (C_{i_1 \dots i_n})$ the cylinder of length $n$ for $T$. We will make use of the relation between the symbolic model and the repeller in order to describe the thermodynamic formalism for the map $T$. We first define the two classes of potentials that we will consider,
 \begin{defi} The class of \emph{regular} potentials is defined by
 \begin{equation*}
\mathcal{R}:= \left\{ \phi : \Lambda \to \R:\phi<0\text{, }\phi\circ\pi\text{ has summable variations and }\lim_{x \to 0}  \phi(x)=-\infty \right\}.
\end{equation*} 
 \end{defi}
  Note that if we have a potential $\psi:\Lambda\to\R$ such that $a\psi+b\in\mathcal{R}$ for some $a,b\in\R$ then since we can compute the Birkhoff spectrum for $a\psi+b\in\mathcal{R}$ we can compute the Birkhoff spectrum for $\psi$.  
 
 \begin{defi}\label{Holder} The class of \emph{strongly regular} potentials is defined by
 \begin{equation*}
\bar{\mathcal{R}}:= \left\{ \phi : \Lambda \to \R: \phi\in\mathcal{R}\text{ and }\phi\circ\pi\text{ is weakly H\"{o}lder} \right\}.
\end{equation*} 
 \end{defi}

 \begin{eje}
Let $\{a_n\}_n$ be a sequence of real numbers such that $a_n \to - \infty$. The  locally constant potential $\phi:\Lambda \to \R$ defined by $\phi(x) = a_n$ if $x \in I(n)$, is such that $\phi \in \bar{\mathcal{R}}$.
 \end{eje}
 
 The \emph{topological pressure} of a potential $\phi \in \mathcal{R}$ is defined by
 \[P_T(\phi)=  \sup \left\{ h(\mu) +\int  \phi \, d \mu : -\int  \phi \, d
\mu < \infty \textrm{ and } \mu\in \mathcal{M}_T  \right\}, \]
where $ \mathcal{M}_T$ denotes the space of $T-$invariant probability measures.
Since there exists a bijection between the space of $\sigma-$invariant measure $\mathcal{M}_{\sigma}$ and the space of  $T-$invariant measures $\mathcal{M}_T$  we have that
\begin{equation}
P_T(\phi)= P( \pi \circ \phi).
\end{equation}
Therefore, all the properties described in Subsection \ref{CMS} can be translated into properties of the topological pressure of the map $T$. Since both pressures have the exact same behaviour, for simplicity, we will denote them both by $P(\cdot)$. 

\begin{rem}
Since we are assuming that the set $E$ of end points of the partition has only one accumulation point and it is zero, we have that if $\phi \in \mathcal{R}$  then $\lim_{x \to 0} \phi(x) = -\infty$ and if $a \in \Lambda \setminus \{0\}$ then
$\lim_{x \to a} \phi(x) < \infty$. 
\end{rem}

\begin{rem} 
Note that if $T$ is an EMR map then the potential $-\log |T'| \in \mathcal{R}$. If $\mu \in \mathcal{M}_T$ then  the integral
\[\lambda(\mu):= \int \log |T'| \, d\mu, \]
will be called the \emph{Lyapunov exponent} of $\mu$.
\end{rem} 
%
%\begin{defi}
%Let $T$ be an EMR map and $\phi \in \mathcal{R}$. A measure $\mu \in \mathcal{M}_T$ is called \emph{maximising} measure if
%\begin{equation*}
%\int \phi \ d \mu = \sup \left\{ \int \phi \ d \nu  : \nu \in \mathcal{M}_T     \right\}.
%\end{equation*}
%Analogously, we  define a \emph{minimising } measure.
%\end{defi}

\subsection{Hausdorff Dimension} \label{hausdorff}
In this subsection we recall basic definitions  from dimension theory. We refer to the books \cite{ba, fa, pe} for further details. A countable collection of sets $\{U_i \}_{i\in N}$ is called a $\delta$-cover of $F \subset\R$ if $F\subset\bigcup_{i\in\N} U_i$, and  for every $i\in\N$ the sets $U_i$ have diameter $|U_i|$ at most $\delta$. Let $s>0$, we define
\[
 H^s_{\delta} (F) :=\inf \left\{ \sum_{i=1}^{\infty} |U_i|^s : \{U_i \}_i \text{ is a } \delta\text{-cover of } F \right\}
\]
and
\[ H^s(F):=  \lim_{\delta \to 0} H^s_{\delta} (F).\]
The \emph{Hausdorff dimension} of the set $F$ is defined by
\[
{\dim_H}(F) := \inf \left\{ s>0 : H^s(F) =0 \right\}.
\]
We will also define the \emph{Hausdorff dimension} of a probability measure $\mu$ by
\[
{\dim_H}(\mu) := \inf \left\{ {\dim_H}(Z): \mu(Z)=1 \right\}.
\]
A measure $\mu \in \mathcal{M}_T$ is called a \emph{measure of maximal dimension} if
$\dim_H \mu = \dim_H \Lambda$.

\section{Variational principle for the Hausdorff dimension}
In this section we prove our main result. That is, we establish the Hausdorff dimension of the level sets $J(\alpha)$ satisfy a conditional variational principle.

\begin{teo} \label{main'}
Let $\phi \in \mathcal{R}$  then for $\alpha\in (-\infty,\alpha_M)$
\begin{equation}
\dim_H(J(\alpha)) = \sup \left\{ \frac{h(\mu)}{\lambda(\mu)} :\mu \in \mathcal{M}_T, \int  \phi \, d \mu = \alpha \textrm{ and } \lambda(\mu) < \infty \right\}.
\end{equation}
\end{teo}

\subsection*{Proof of the lower bound}
In order to prove the lower bound first note that if $ \mu \in \mathcal{M}_T$ is ergodic and
$\int  \phi \, d \mu = \alpha$ then $\mu(J(\alpha))=1$. Moreover if $\lambda(\mu) < \infty$ then $\dim_H(\mu)=\frac{h(\mu)}{\lambda(\mu)}$ and we can conclude that
\[\dim_H(J(\alpha)) \geq \dim_H(\mu) = \frac{h(\mu)}{\lambda(\mu)}.\]
Thus we can deduce that
\begin{equation*}
\dim_H(J(\alpha)) \geq \sup \left\{ \frac{h(\mu)}{\lambda(\mu)} :\mu \in \mathcal{M}_T\text{ and ergodic }, \int  \phi \, d \mu = \alpha \textrm{ and } \lambda(\mu) < \infty \right\}.
\end{equation*}
To complete the proof of the lower bound we need the following lemma
\begin{lema}
Let $\alpha\in (-\infty, \alpha_{M})$.
If $ \mu \in \mathcal{M}_T$, $\int \phi \,  d \mu=\alpha$ and $\lambda(\mu)<\infty$ then for any $\epsilon>0$ we can find $\nu\in\mathcal{M}_T$ which is ergodic and 
\begin{enumerate}
\item
$\int\phi \,  d  \nu=\alpha$,
\item
$|h(\nu)-h(\mu)|\leq\epsilon$,
\item
$|\lambda(\nu)-\lambda(\mu)|\leq\epsilon$.
\end{enumerate}
\end{lema}
\begin{proof}
Let $ \mu \in \mathcal{M}_T$, $\int \phi \,  d \mu=\alpha$ and $\lambda(\mu)<\infty$. We can then find a sequence of invariant measures $\{  \mu_n \}$ supported on finite subsystems such that $\int\phi \,  d \mu_n=\alpha$, $\lim_{n\rightarrow\infty}\lambda(\mu_n)=\lambda(\mu)$ and $\lim_{n\rightarrow\infty} h(\mu_n)=h(\mu)$. Since these measures are supported on finite subsystems we can apply Lemma 2 and Lemma 3 from \cite{jjop} to complete the proof.  
\end{proof} 
We can now immediately deduce that
\begin{eqnarray*}
 \sup \left\{ \frac{h(\mu)}{\lambda(\mu)} :\mu \in \mathcal{M}_T, \int  \phi \, d \mu = \alpha \textrm{ and } \lambda(\mu) < \infty \right\} =\\
  \sup \left\{ \frac{h(\mu)}{\lambda(\mu)} :\mu \in \mathcal{M}_T \textrm{ and ergodic } , \int  \phi \, d \mu = \alpha \textrm{ and } \lambda(\mu) < \infty \right\},
  \end{eqnarray*}
 which completes the proof of the lower bound.
  \subsection{Upper bound}
In this section we prove the upper bound of our main result. We adapt to our setting the method used in \cite{jjop}.

\begin{lema}\label{continuity}
The function
\begin{equation*}
F(\alpha):=\sup \left\{ \frac{h(\mu)}{\lambda(\mu)} :\mu \in \mathcal{M}_T, \int  \phi \, d \mu = \alpha \textrm{ and } \lambda(\mu) < \infty \right\}
\end{equation*}
is continuous in the domain $(-\infty,\alpha_M)$.
\end{lema}

\begin{proof}
 Let $\{ \mu_n \}$ be a sequence of measures in $\mathcal{M}_T$ satisfying $\lambda(\mu_n)<\infty$ and converging to a measure $\mu$ where $\int \phi \,  d \mu=\alpha$. Let $\overline{\mu},\underline{\mu}\in {M}_T$ such that 
$$\int\phi \,  d \underline{\mu}<\alpha<\int\phi \,  d \overline{\mu}$$
and $\lambda(\overline{\mu}),\lambda(\overline{\mu})<\infty$. By considering convex combinations of $\mu_n$ with $\overline{\mu}$ or $\underline{\mu}$ we can find a sequence of measures $\nu_n$ where $\int \phi \,  d \nu_n=\alpha$ for each $n$ and
$$\lim_{n\rightarrow\infty}\left|\frac{h(\mu_n)}{\lambda(\mu_n)}-\frac{h(\nu_n)}{\lambda(\nu_n)}\right|=0.$$
It then follows that
\[ F(\alpha) \geq \limsup_{n \to \infty} F(\alpha_n).\]
In the other direction we fix $\mu, \nu\in\mathcal{M}_T$ with $\int\phi \,  d \nu=\beta<\alpha=\int\phi \,  d \mu$. Let $\nu_p=p\nu+(1-p)\mu$ and note that 
$$\liminf_{x\rightarrow\alpha^{-}} F(x)\geq\lim_{p\rightarrow 0} \frac{h(\nu_p)}{\lambda(\nu_p)}=\frac{h(\mu)}{\lambda(\mu)}$$
and
$$\liminf_{x\rightarrow\beta^{+}} F(x)\geq\lim_{p\rightarrow 1} \frac{h(\nu_p)}{\lambda(\nu_p)}=\frac{h(\nu)}{\lambda(\nu)}.$$
We can use this to deduce that 
$$ F(\alpha) \leq \liminf_{n \to \infty} F(\alpha_n).$$
\end{proof}

Denote by $S_k \phi(x):= \sum_{i=0}^{k-1} \phi(T^i x)$.
Let $\alpha \in \R, N \in \N$ and $ \epsilon>0$ and consider the following set,
\begin{equation} \label{jotas}
J(\alpha, N, \epsilon):= \left\{x \in \Lambda: \frac{S_k \phi(x)}{k} \in \left(\alpha - \epsilon, \alpha + \epsilon \right), \textrm{ for every } k \geq N   \right\}.
\end{equation}

Note that
\begin{equation*}
J(\alpha) \subset \bigcup_{N=1}^{\infty} J(\alpha, N, \epsilon).
\end{equation*}

In order to obtain an upper bound on the dimension of $J(\alpha)$ we will compute upper bounds on the dimension of $J(\alpha, N, \epsilon)$. Denote by $\mathcal{C}_k$ the cover of $J(\alpha, N, \epsilon)$ by cylinders of length $k \in \N$, that is
\[\mathcal{C}_k:= \left\{ I(i_1, \dots, i_{k}) :  I(i_1, \dots, i_{k}) \cap  J(\alpha, N, \epsilon)  \neq \emptyset \right\}. \]

\begin{lema}
For every $k \in \N$ the cardinality of $\mathcal{C}_k$ is finite.
\end{lema}

\begin{proof}
Since $\phi \in \mathcal{R}$ we can deduce that $\lim_{i\rightarrow\infty}\inf_{x\in I(i)} \phi(x)=-\infty$ and hence we can find an $i\in\N$ such that for all $x\in I(j)$ with $j\geq i$ we have that $|\phi(x)|>k(\alpha+\epsilon)$. It then follows that $C_k$ only contains cylinders $I(i_1,\ldots,i_k)$ where each $i_l<i$. There is clearly only a finite number of such cylinders. 
\end{proof}

Let $s_k \in \R$ denote the unique real number such that
\begin{equation*}
\sum_{I(i_1, \dots i_k) \in \mathcal{C}_k} | I(i_1, \dots i_k) |^{s_k} = 1.
\end{equation*}
We define the following number:
\begin{equation}
s:= \limsup_{k \to \infty} s_k
\end{equation}

\begin{lema}\label{cover}
The following bound holds,
\[\dim_H (J(\alpha, N, \epsilon)) \leq s, \]
and there exists a sequence of $T-$invariant probability measures $\{\mu_k\}$ such that
\[ \lim_{k \to \infty} \left( s_k - \frac{h(\mu_k)}{\lambda(\mu_k)}             \right)    =0           \]
and $\int\phi\, d \mu_k\in  (\alpha-2\epsilon,\alpha+2\epsilon)$.
\end{lema}

\begin{proof}
To see that $\dim_H (J(\alpha, N, \epsilon)) \leq s, $ we note that for $k$ sufficiently large and $\epsilon>0$
$$H_{\xi^k}^{s+\epsilon}(J(\alpha, N, \epsilon))\leq \sum_{I(i_1, \dots i_k) \in \mathcal{C}_k} | I(i_1, \dots i_k) |^{s+\epsilon}\leq 1.$$
This means that $H^{s+\epsilon}(J(\alpha, N, \epsilon))\leq 1$ and so $\dim J(\alpha, N, \epsilon)\leq s+\epsilon$.

For the second part let $\eta_k$ be the $T^k$-invariant Bernoulli measure which assigns each cylinder in $C_k$, denoted by $I(i_1,\ldots,i_k)$, the probability $|I(i_1,\ldots,i_k)|^{s_k}$. Note that the entropy of this measure with respect to $T^k$ will be
$$h(\eta_k,T^k)=-s_k\sum_{I(i_1,\ldots,i_k)\in C_k} |I(i_1,\ldots,i_k)|^{s_k}\log |I(i_1,\ldots,i_k)|$$ 
and there will exist $C>0$ such  that for all $k\in\N$ the Lyapunov exponent $\lambda(\eta_k, T^{k+1})$ satisfies 
$$\left|-\lambda(\eta_k,T_k)-\sum_{I(i_1,\ldots,i_k)\in C_k} |I(i_1,\ldots,i_k)|^{s_k}\log |I(i_1,\ldots,i_k)|\right|\leq C.$$
This then gives that
$$\frac{s_k(\lambda(\eta_k,T^k)-C)}{\lambda(\eta_k,T^k)}\leq\frac{h(\mu,T^k)}{\lambda(\eta_k,T^k)}\leq\frac{s_k(\lambda(\eta_k,T^k)+C)}{\lambda(\eta_k,T^k)}$$
and since $\lambda(\eta_k,T^k)\geq\xi^{k}$ it follows that $\lim_{k\rightarrow\infty}\frac{h(\eta_k,T^k)}{\lambda(\eta_k,T^k)}-s_k=0$. Moreover, for $k$ sufficiently large each cylinder in $C_k$ will only contain points $x$ where $S_k\phi(x)\in (\alpha-2\epsilon,\alpha+2\epsilon)$. This means that $\int\frac{S_k\phi}{k}\, d \eta_k\in (\alpha-2\epsilon,\alpha+2\epsilon)$. To complete the proof we simply let $\mu_k=\sum_{i=0}^{k-1}\eta_k\circ T^{-i}$.  
\end{proof}
Thus, we can deduce that 
$$\dim_H J(\alpha)\leq\lim_{\epsilon\rightarrow 0}\sup_{\gamma\in (\alpha-\epsilon,\alpha+\epsilon)} F(\gamma).$$ 
The fact that
$$\dim_H J(\alpha)\leq F(\alpha)$$
now follows by Lemma \ref{continuity}. This completes the proof of Theorem \ref{main'}.

\begin{rem}
It is a direct consequence of the work of Barreira and Schmeling \cite{b.s} together with the approximation property of the pressure (Proposition  \ref{app}) that the irregular set has full Hausdorff dimension,
\begin{equation*}
\dim_H J'= \dim_H \Lambda.
\end{equation*}
\end{rem}

\section{Regularity of the multifractal spectrum} \label{R}
This section is devoted to the study of the regularity properties of the multifractal spectrum.  We relate the conditional variational principle to thermodynamic properties and as a result prove Theorem \ref{analytic} . Our proof is based on ideas developed by Barreira and Saussol  \cite{bs1} in the uniformly hyperbolic  (Markov with finitely many branches) setting. Nevertheless, most of their arguments can not be translated into the non-compact (Markov with countably many branches) setting. It should be pointed out that the behaviour of the multifractal spectrum in this setting is much richer than in the compact setting. New phenomena occurs, in particular the multifractal spectrum can be constant and  it can have points where it is not analytic. We obtain conditions ensuring these new phenomena to happen.

The following Lemma is a direct consequence of results by Mauldin and Urba\'nski \cite{mu}, Sairg \cite{sa2} and Stratmann and Urba\'nski \cite{su}. We will use it to deduce certain regularity properties of the multifractal spectrum. Throughout this section we will let $\phi\in\mathcal{R}$ and $\alpha_M$ to be as in the introduction. Some of the results will need additionally that $\phi\in\bar{\mathcal{R}}$.
\begin{prop}[Regularity] \label{r}
If $\phi\in\bar{\mathcal{R}}$, $\delta \in (0,1]$ and $\alpha \in (-\infty, \alpha_M)$ then the function
\[q \mapsto P( q(\phi - \alpha) - \delta \log|T'|), \]
when finite is real analytic, and in this case
\begin{equation}
\frac{d}{dq}  P( q(\phi - \alpha) - \delta \log|T'|) \Big|_{q=q_0}= \int \phi \, d \mu_{q_0, \delta}  - \alpha,
\end{equation}
where $\mu_{q_0, \delta}$ is the equilibrium state of  the potential $ q_0(\phi - \alpha) - \delta \log|T'|$.
\end{prop}
For $\alpha\in (-\infty,\alpha_M)$ we will let
$$\delta(\alpha)=\sup \left\{ \frac{h(\mu)}{\lambda(\mu)} :\mu \in \mathcal{M}_T, \int  \phi \, d \mu = \alpha \textrm{ and } \lambda(\mu) < \infty \right\}.$$
We wish to relate $\delta(\alpha)$ to the function $q \mapsto P( q(\phi - \alpha) - \delta \log|T'|)$. To do this we introduce the value $\delta^*$ which is defined by
$$\delta^*:=\inf \left\{\delta\in [0,1]:P(q\phi-\delta\log |T'|)<\infty\text{ for some }q> 0 \right\}.$$
This quantity will alway give a lower bound for $\delta(\alpha)$.

\begin{lema}\label{lb}
For all $\alpha\in (-\infty,\alpha_M)$ we have that $\delta(\alpha)\geq\delta^*$. 
\end{lema}

\begin{proof}
 If $\delta^*=0$ then this statement is obvious so we will assume that $\delta^*>0$. Let $0<s<\delta^*$ and $\alpha\in (-\infty,\alpha_M)$.   
 In order to show that $\delta(\alpha) > \delta^*$ we will exhibit a sequence of invariant measures $(\nu_n)$ such that for every $n \in \N$ we have $\int \phi \ d \nu_n = \alpha$ and
 \[\lim_{n \to \infty} \frac{h(\nu_n)}{\lambda(\nu_n)} \geq s.\]
 First note that we can find a sequence of invariant measures $(\mu_n)$ such that for all $n$ we have $s\lambda(\mu_n)<h(\mu_n)<\infty$ and $\lim_{n\rightarrow\infty}\frac{h(\mu_n)}{-\int\phi\text{d}\mu_n}=\infty$. Indeed, note that for every $q>0$ we have that $P(q\phi-\delta\log |T'|) = \infty$. Let $q >0$ and $A>0$ with $A>q\alpha_M$. Because of the approximation property of the pressure,  we can choose an invariant measure $\nu$ satisfying
 \begin{equation} \label{A}
 h(\nu) +q \int \phi \ d \nu -s \lambda(\nu) \geq A. 
 \end{equation}
 That is
 \[h(\nu)> \left(A- q\alpha_M	\right) + s \lambda(\nu).\]
  From where we can deduce that
 \[ s \lambda(\nu)<h(\nu) < \infty.\]
Since $\int \phi \ d \nu <0$ then from equation \eqref{A} we have
 \begin{equation} \label{B}
 \frac{h(\nu)}{-\int \phi  \ d  \nu} >  -\frac{A}{\int \phi  \ d  \nu} -s  \frac{\lambda(\nu)}{\int \phi  \ d  \nu} +q > q.
  \end{equation}
Since we can do this for every positive $q \in \R$, let $q=n$ and denote by $\mu_n$ an invariant measures satisfying equations \eqref{A} and \eqref{B}. The sequence $(\mu_n)$ complies  with the required conditions.

Passing to a subsequence if necessary, we can assume that the sequence $\int\phi\text{d}\mu_n$  is monotone  and that the following limit exists $\gamma=\lim_{n\rightarrow}\int\phi\text{d}\mu_n$ (note that $\gamma$ can be $-\infty$). 

 For sufficiently large values of $n \in \N$ the integral $\int\phi\text{d}\mu_n$ is close to $\gamma$.  Therefore, there exists $\beta \in \R$ and an invariant measure $\mu$ satisfying: 
 \begin{enumerate}
 \item $\int\phi\text{d}\mu=\beta$,
 \item $h(\mu) < \infty$ and $\lambda(\mu) <\infty$,
 \item $\alpha\in (\beta,\int\phi\text{d}\mu_n]$ or $\alpha\in [\int\phi\text{d}\mu_n,\beta)$  for $n \in \N$ large enough. 
  \end{enumerate}

For $n$ sufficiently large we can also find constants $p_n \in [0,1]$ such that $\alpha=p_n\beta+(1-p_n)\int\phi\text{d}\mu_n$. If $p_n=0$ for all $n$ sufficiently large then there is nothing to prove. Consider the following sequence of invariant measures $(\nu_n)$ defined by
\[ \nu_n = p_n \mu +(1-p_n) \mu_n.\]
Then $\int \phi \ d \nu_n = \alpha$. By construction we have that $\lim_{n \to \infty} h(\mu_n)= \infty$. Since by assumption $\alpha \neq \beta$ we have that 
$\lim_{n \to \infty} (1-p_n) \in (0, 1]$. Therefore
\[\lim_{n\rightarrow\infty} (1-p_n)h(\mu_n)=\infty.\]
This implies that
$$\lim_{n \to \infty} \frac{h(\nu_n)}{\lambda(\nu_n)}  =\lim_{n\rightarrow\infty}\frac{ p_n h(\mu)+(1-p_n)h(\mu_n)}{p_n\lambda(\mu)+(1-p_n)\lambda(\mu_n)}\geq s.$$
    \end{proof}

For notational ease we will allow $P( q(\phi - \alpha) - \delta \log|T'|)\geq 0$ to include the case when it is infinite.

\begin{lema}\label{positive}
If $\phi\in\mathcal{R}$,  $\alpha \in (-\infty, \alpha_M)$ and $\delta(\alpha)>\delta^*$  
then for all $q\in\R$ we have
$$P( q(\phi - \alpha) - \delta(\alpha) \log|T'|)\geq 0$$
\end{lema}

\begin{proof}
Recall that
$$\delta(\alpha)=\sup \left\{ \frac{h(\mu)}{\lambda(\mu)} :\mu \in \mathcal{M}_T, \int  \phi \, d \mu = \alpha \textrm{ and } \lambda(\mu) < \infty \right\}.$$
Denote by $(\mu_n)_n$ a sequence of $T-$invariant measures such that for every $n \in \N$ we have
\begin{enumerate}
\item $\int \phi \ d \mu_n= \alpha$,
\item $h(\mu_n) < \infty$ and $\lambda(\mu_n) < \infty$, 
\item \[ \lim_{n \to \infty} \frac{h(\mu_n)}{\lambda(\mu_n)} = \delta(\alpha).\]
\end{enumerate}
If we choose $\delta^*<s<\delta(\alpha)$ and $q_0>0$ such that $P(q_0\phi-s\log |T'|)=K<\infty$ then by the variational principle for all $n$ we have
$$q\int\phi\text{d}\mu_n-s\lambda(\mu_n)+h(\mu_n)\leq K.$$
Since for $n$ sufficiently large we have
\[  \delta^* < \frac{h(\mu_n)}{\lambda(\mu_n)} = \delta(\alpha),\]
we obtain $\delta^* \lambda(\mu_n) < h(\mu_n)$. Thus, for $n$ sufficiently large we have that
$$\frac{\delta^*-s}{2}\lambda(\mu_n)\leq   \delta^*\lambda(\mu_n) - s \lambda(\mu_n) \leq
h(\mu_n) -s \lambda(\mu_n) \leq  K-q\alpha.$$
Furthermore by the variational principle we have
\begin{eqnarray*}
P( q(\phi - \alpha) - \delta(\alpha) \log|T'|)\geq h(\mu_n)+ q \left( \int \phi \ d \mu_n - \alpha \right)-\delta(\alpha) \lambda(\mu_n) =  &\\ h(\mu_n)- \delta(\alpha) \lambda(\mu_n) \geq 
\lambda(\mu_n) \left( \frac{h(\mu_n)}{\lambda(\mu_n)}  -\delta(\alpha)  \right).
 \end{eqnarray*}
The result then follows since
\begin{equation*}
\liminf_{n \to \infty} \left(\lambda(\mu_n) \left( \frac{h(\mu_n)}{\lambda(\mu_n)}  -\delta(\alpha)  \right) \right) \geq 0.
\end{equation*}

\end{proof}

We can now describe the function $q\to P(q(\phi-\alpha)-\delta(\alpha)\log |T'|)$ in more detail.

\begin{lema}\label{options}
For any $\alpha\in (-\infty,\alpha_M]$ one of the following three statements will hold,
\begin{enumerate}
\item\label{case1}
$\delta(\alpha)=\delta^*.$
\item\label{case2}
There exists $q_0 \in \R$ such that $P\left(q_0(\phi-\alpha)-\delta(\alpha) \log |T'|\right)=0$ and $$\frac{\partial}{\partial q}P\left(q(\phi-\alpha)-\delta(\alpha) \log |T'| \right) \Big|_{q={q_0}}=0.$$
\item\label{case3}
There exists $q_c \in \R$ such that $P\left(q_c(\phi-\alpha)-\delta(\alpha) \log|T'|\right)=0$ and $$P\left(q(\phi-\alpha)-\delta(\alpha)\log |T'|\right)=\infty$$ for all $q<q_c$. 
\end{enumerate}
\end{lema}

\begin{proof}
We will assume throughout that $\delta(\alpha)>\delta^*$ since otherwise (\ref{case1}) is satisfied. 

We know that when finite the function $q\to P\left(q(\phi-\alpha)-\delta(\alpha)\log |T'|\right)$ is real analytic. Moreover, in virtue of Lemma \ref{positive}, for all $q \in \R$ we have
 $$P\left(q(\phi-\alpha)-\delta(\alpha)\log |T'|\right)\geq 0.$$ We will show that if the derivative of the pressure is zero then the pressure itself is also zero. Indeed, assume that there exists $q_0 \in \R$ such that
$$\frac{\partial}{\partial q}P\left(q(\phi-\alpha)-\delta(\alpha)\log |T'|\right) \Big|_{q={q_0}}=0.$$
Denote by $\mu_{q_0}$ the equilibrium measure corresponding to the potential
$q_0(\phi-\alpha)-\delta(\alpha)$. Then, Ruelle's formula for the derivative of pressure gives that $\int\phi\text{d}\mu_{q_0}=\alpha$. Thus
$$P\left(q(\phi-\alpha)-\delta(\alpha)\log |T'|\right)=-\delta(\alpha)\lambda(\mu_{q_0})+h(\mu_{q_0})\leq 0.$$
So, $P\left(q_0(\phi-\alpha)-\delta(\alpha)\log |T'|\right)=0$ and statement \ref{case2} holds. Note that if the pressure function  $q\to P\left(q(\phi-\alpha)-\delta(\alpha)\log |T'|\right)$ is finite for every $q \in \R$ then there must exists $q_0 \in \R$ such that the derivative of
$P\left(q(\phi-\alpha)-\delta(\alpha)\log |T'|\right)$ at $q=q_0$ is equal to zero. This follows from
 Ruelle's formula for the derivative of pressure and the fact that $\alpha \in (-\infty, \alpha_M)$.
 
Let us assume now that the derivative of the pressure does not vanish at any point
 and let $q_c=\inf\{q:P(q(\phi-\alpha)-\delta(\alpha)\log |T'|)<\infty\}$. It follows from standard ergodic optimization arguments \cite{jmu,l} that 
$$\lim_{q*\to\infty}\frac{\partial}{\partial q} P(q(\phi-\alpha)-\delta(\alpha) \log |T'|) \Big|_{q={q*}}>0.$$ 
If $P(q_c(\phi-\alpha)-\delta(\alpha)\log |T'|)=\infty$ then by considering compact approximations to the pressure we can see that 
$$\lim_{q\to q_c^+}P(q(\phi-\alpha)-\delta(\alpha)\log |T'|)=\infty.$$
But recall that for $q >q_c$  the pressure is finite. 
 This means that for small $\epsilon >0$ the derivative of the pressure  for $q \in (q_c, q_c + \epsilon)$ will be negative.  This, in turn,  will imply that there is a zero for the derivative and so cannot happen. Thus
$P(q_c(\phi-\alpha)-\delta(\alpha)\log |T'|)<\infty$ and 
$$\frac{\partial}{\partial q} P(q_c(\phi-\alpha)-\delta(\alpha)\log |T'|) \Big|_{q=q_c}>0.$$ 

If $P(q_c(\phi-\alpha)-\delta(\alpha)\log |T'|)=C>0$ then there exits a compact invariant set $K$ on which the pressure restricted to $K$ satisfy $P_K(q(\phi-\alpha)-\delta(\alpha)\log|T'|)>0$ for all $q \in \R$. By considering the behaviour as $q\to\infty$ and $q\to-\infty$ this function must have a critical  point that we denote by $q_K$. denote by $\mu_K$  the equilibrium measure corresponding to $q_K(\phi-\alpha)-\delta(\alpha)\log |T'|$.
 We can conclude that $\int\phi\text{d}\mu_K=\alpha$ and so
$$0<P_K(q_K(\phi-\alpha)-\delta(\alpha)\log |T'|)= h(\mu_K) -\delta(\alpha) \lambda(\mu_K).$$
This means that $h(\mu_K)/\lambda(\mu_K)>\delta(\alpha)$ which contradicts the definition of $\delta(\alpha)$. So we can conclude that 
$$P(q_c(\phi-\alpha)-\delta(\alpha)\log |T'|)=0$$
and Property \ref{case3} is satisfied. 
\end{proof}

We will denote by $A(\alpha)$ the set of values $\alpha \in (-\infty, \alpha_M)$ where case \ref{case2} of Lemma \ref{options} is satisfied.

\begin{lema}\label{analytic2} Let $I\subset A(\alpha)$ be an interval. The function $\alpha\rightarrow b(\alpha)=\delta(\alpha)$ is real analytic on I.
\end{lema}

\begin{proof}
Recall that 
\begin{equation*}
b(\alpha)=  \sup \left\{ \frac{h(\mu)}{\lambda(\mu)} :\mu \in \mathcal{M}_T, \int  \phi \, d \mu = \alpha \textrm{ and } \lambda(\mu) < \infty \right\}.
\end{equation*}
In virtue of the definition of $I$ we have that for $\alpha\in I$ there exists $q(\alpha) \in \R$ such that
\begin{equation*}
P( q(\alpha)(\phi - \alpha) - b(\alpha) \log|T'|) =0.
\end{equation*}
Recall that the function $(q,\delta) \to P( q(\phi - \alpha) - \delta \log|T'|)$ is real analytic on each variable. In order to obtain the regularity of $b(\alpha)$ we will apply the implicit function theorem.  Proceeding as in Lemma $9.2.4$ of \cite{ba}, if
\begin{equation*}
G(q, \delta, \alpha):= \left( \begin{array}{c} P(q(\phi- \alpha) -\delta \log |T'|) \\ \frac{\partial P(q(\phi- \alpha) -\delta \log |T'|)}{\partial q} \end{array} \right) 
\end{equation*}
we just need to show that
\begin{eqnarray*}
\det \left[  \left( \frac{\partial G}{\partial q},   \frac{\partial G}{\partial \delta}            \right) \right]= \frac{\partial P(q(\phi- \alpha) -\delta \log |T'|) }{\partial q} \cdot \frac{ \partial^2 P(q(\phi- \alpha) -\delta \log |T'|)} {\partial \delta \partial q} - \\
 \frac{\partial^2  P(q(\phi- \alpha) -\delta \log |T'|) }{ \partial q^2} \cdot \frac{\partial P(q(\phi- \alpha) -\delta \log |T'|) }{\partial \delta} ,
  \end{eqnarray*}
is not equal to zero for $\delta= b(\alpha)$ and $q=q(\alpha)$. Since $\partial P(q(\phi- \alpha) -\delta \log |T'|)  / \partial q =0$ at $q=q(\alpha)$ it is sufficient to show that $ \partial^2 ( P(q(\phi- \alpha) -\delta \log |T'|)) / \partial q^2 $ and $ \partial ( P(q(\phi- \alpha) -\delta \log |T'|))  / \partial \delta$ are nonzero. 
Since the function $P(q(\phi - \alpha) -\delta \log |T'|)$ is strictly convex as a function of the variable $q$ we have that
\[\frac{ \partial^2 ( P(q(\phi- \alpha) -\delta \log |T'|))} { \partial q^2 } \neq 0.\]
Since, there exists an ergodic equilibrium measure $\mu_e$ such that
\[ \frac{\partial P(q(\phi - \alpha) -\delta \log |T'|)}{\partial \delta} = -\int \log|T'| d \mu_{e},\]
then we have
\[ \frac{\partial P(q(\phi - \alpha) -\delta \log |T'|)}{\partial \delta}  <0.\]
Therefore the function $b(\alpha)$ is real analytic on I.
\end{proof}

Let $s_{\infty}=\inf  \left\{ s \in \R : P(-s \log |T'| ) < \infty \right\}.$ We are now ready to complete the proof of Theorem \ref{analytic} with the following more general proposition.

\begin{prop}
Let $\phi \in \mathcal{R}$. We have that
\begin{enumerate}
\item[1.]
If $\delta^*=\dim_H \Lambda$ then $b(\alpha)=\delta^*$ for all $\alpha\in (-\infty,\alpha_M]$.
\item[2.]
If $\delta^*\leq s_{\infty}<\dim\Lambda$ then there exists a non-empty interval $I$ for which $I\subset A(\alpha)$ and thus $b(\alpha)$ is analytic for a region of values of $\alpha$. 
\item[3.]
If $\lim_{x\to 0}\frac{-\phi(x)}{\log |T'(x)|}=\infty$ then either
\begin{enumerate}
\item
$A(\alpha)=(-\infty,\alpha_M]$ and thus $b(\alpha)$ is analytic for $\alpha\in(-\infty,\alpha_M)$ or
\item 
there exists an ergodic measure of full dimension $\nu$ with $\underline{\alpha}=\int\phi\text{d}\nu>-\infty$ and then $I(\alpha)=[\underline{\alpha},\alpha_M]$, $b(\alpha)$ is analytic for $\alpha\in(-\underline{\alpha},\alpha_M]$ and $b(\alpha)=\dim\Lambda$ for $\alpha\leq\underline{\alpha}$.
\end{enumerate}
\item[4.]
If $\lim_{x\to 0}\frac{-\phi(x)}{\log |T'(x)|}=0$  then $b(\alpha)$ is analytic on $(-\infty,\alpha_M]$ except for at most two points.
\end{enumerate}
\end{prop}

\begin{proof} Each part will be proved separately.\\

Part 1 can be immediately deduced from Lemma \ref{lb}. \\
%THERE IS A CASE MISSING, THAT IS THE ONE WHERE $s_{\infty}= \dim_H \Lambda$, TJ: WHY. LEMMA 4.1 COVERS THIS CASE AS WELL.   

To prove part 2 we let $s=\dim\Lambda$ and note that  $\delta^* \leq s_{\infty} < s.$ Since $s_{\infty} < s$ then $P(-s\log |T'|)=0$ and $P(-t\log |T'|)>0$ for $s_{\infty}<t<s$ and $P(-t\log |T'|)=\infty$ for $\delta^*<t<s_{\infty}$. Denote by $\nu$ be the equilibrium state corresponding to $-s\log |T'|$ and $\underline{\alpha}=\int\phi\text{d}\nu$ (this can be $-\infty$, but if finite then $b(\underline{\alpha})=s$). Since $\delta(\alpha)$ is a continuous function of $\alpha$ we can define
$$\overline{\alpha}=\sup\{\alpha:\delta(\alpha)>s_{\infty}\}.$$
Now we assume that $\alpha\in (\underline{\alpha},\overline{\alpha})$ and so in particular  $\delta(\alpha)>\delta^*$. Since $\delta(\alpha)>\delta^*$ we are in either case \ref{case2} or \ref{case3} of Lemma \ref{options}. Therefore there  exist $q_0 \in \R$ such that $P(q_0(\phi-\alpha)-\delta(\alpha)\log |T'|)=0$. Let $q<0$  and note that there exists a compact invariant set $K \subset \Lambda$ and a $T$-invariant measure $\nu_{\alpha}$ such that
$\dim_H \nu_{\alpha} > \delta (\alpha)$ and $\int \phi \ d \nu_{\alpha} < \alpha$.  
We have that
\begin{eqnarray*}
P(q(\phi-\alpha)-\delta(\alpha)\log |T'|)\geq  h(\nu_{\alpha}) + q \left(\int\phi\text{d}\nu_{\alpha} -\alpha \right)-\delta(\alpha)\lambda(\nu_{\alpha}) \\
 = q \left(\int \phi \ d \nu_{\alpha}-\alpha \right)+  \lambda(\nu_{\alpha})  \left (\frac{h(\nu_{\alpha})}{  \lambda(\nu_{\alpha})} - \delta(\alpha) \right)
\geq 0
\end{eqnarray*}
and so $q_0\geq 0$. We also have that $P(-\delta(\alpha)\log |T'|)\geq 0$ with equality if and only if $\alpha=\underline{\alpha}$. By the definition of $\delta^*$ and noticing that $\delta(\alpha) > \delta^*$ there exists $q^*>0$ such that if  $q\in (0,q^*)$ then 
$$P(q(\phi-\alpha)-\delta(\alpha)\log |T'|)<\infty.$$
Thus if $\delta(\alpha)<\delta^*$ then $q\rightarrow P(q(\phi-\alpha)-\delta(\alpha)\log |T'|)$ is decreasing for $q$ sufficiently close to $0$ and we can only be in case \ref{case2} from Lemma \ref{options}. If $\alpha=\underline{\alpha}$ then $P(-\delta(\underline{\alpha})\log |T'|)=0$ and $\frac{\partial}{\partial q}P(q(\phi-\underline{\alpha})-s\log |T'|) \Big|_{q=0}=0$ which means we are also in  case \ref{case2} from Lemma \ref{options}.\\

To prove part 3 we first note that  $\delta^*=0$ .
Indeed, given $A>1$ there exists $\epsilon >0$ such that if $x \in (0, \epsilon)$ then
\begin{equation*}
\frac{-\phi(x)}{\log |T'(x)|} >A, 
\end{equation*}
that is, $\phi(x) < -A \log |T'(x)|$. If we denote by $P_{\epsilon} (\cdot)$ the pressure of 
$T$ restricted to the maximal $T-$invariant set in $(0, \epsilon)$ we have that  $P_{\epsilon} (\phi) \leq P_{\epsilon} (-A \log|T'|) < \infty$. Since the entropy of $T$ restricted to $(0,1) \setminus  (0, \epsilon)$ is finite and the potential $\phi$ restricted to this set is bounded, we can deduce that
$P(\phi)< \infty$. In particular, we obtain that $\delta^*=0.$

Let us consider first the case where $s_{\infty}<s$. In this setting the potential $-s\log |T'|$ has an associated equilibrium state $\nu$ with $h(\nu)/\lambda(\nu)=s$. If we have $\int\phi\text{d}\nu=-\infty$ then we can just apply the techniques from the previous part. If $\int\phi\text{d}\nu:=\underline{\alpha}>-\infty$ then for $\alpha\in (\underline{\alpha},\alpha_M)$ we can see that $b(\alpha)=\delta(\alpha)$ will be analytic by applying part 2. 
 For $\alpha<\underline{\alpha}$ we know for $0<\delta \leq s$
\begin{enumerate}
\item
$P(q(\phi-\alpha)-\delta\log|T'|)=\infty$ for all $q<0$,
\item
$P(-\delta\log |T'|)>0$,
\item
$P(q(\phi-\alpha)-\delta\log |T'|)\geq q(\underline{\alpha}-\alpha)-\delta\lambda(\nu)+h(\nu)>0$ for all $q>0 $. 
\end{enumerate}
Note that the first statement follows from  the assumption $\lim_{x\to 0}\frac{-\phi(x)}{\log |T'(x)|}=\infty$. Therefore we cannot be in cases \ref{case1} or \ref{case2} from Lemma \ref{options}. This means that we must be in case \ref{case3} from Lemma \ref{options} with $q_c=0$ and thus $\delta(\alpha)=s$.

We now assume that $s=s_{\infty}$. We start with the case where $P(-s\log |T'|)<0$. This means that if $\{\mu_n\}_{n\in\N}$ is a sequence of $T$-invariant measures such that $\lim_{n\rightarrow\infty}\frac{h(\mu_n)}{\lambda(\mu_n)}=s$ then $\lim_{n\rightarrow\infty}\lambda(\mu_n)=\infty$.  Indeed, assume by way of contradiction that  $\limsup_{n \to \infty} \lambda(\mu_n)=L< \infty$. Given $\epsilon >0$ there exists $N \in \N$ such that
\begin{equation*}
\frac{h(\mu_N)}{\lambda(\mu_N)} > s - \epsilon,
\end{equation*}
that is $h(\mu_N) -s \lambda(\mu_N) \geq -\epsilon L$. Since this holds for arbitrary values of $\epsilon$ we obtain that $P(-s \log |T'|) \geq 0$. This contradiction proves the statement. Now, since by assumption  $\lim_{x\to 0}\frac{-\phi(x)}{\log |T'(x)|}=\infty$ we have that  $\lim_{n\rightarrow\infty}\int \phi\text{d}\mu_n=-\infty$. Therefore for $\alpha\in (-\infty,\alpha_M)$ we must have  $\delta(\alpha)<s$. This means that $P(-\delta(\alpha)\log|T'|)=\infty$ and for all $q>0$ we have $P(q(\phi-\alpha)-\delta(\alpha)\log |T'|)<\infty$. Thus we must be in Case 2 from Lemma \ref{options} and the proof is complete.

We now assume that $P(-s\log|T'|)=0$ and that $\nu$ is the equilibrium state for $-s\log |T'|$. For any $\alpha\leq \int \phi\text{d}\nu$ we can argue exactly as when $s_{\infty}<s$ to show that $\delta(\alpha)=s$. For $\alpha>\int \phi\text{d}\nu$ we first need to show that $\delta(\alpha)<s$. To prove this, note that the function $q\to P(q(\phi-\alpha)-s\log |T'|)$ has a one sided derivative at $q=0$ with derivative $\int(\phi-\alpha)\text{d}\nu<0$. Thus by Lemma \ref{options} it is not possible that $\delta(\alpha)=s$. So $\delta(\alpha)<s$ and we can use the same arguments as when $P(-s\log |T'|)<0$.\\

We now turn to part 4 of the Lemma. In this case $\delta^*=s_{\infty}$. Indeed, given $t>0$ there exist $\epsilon >0$ such that if
$x \in (0, \epsilon)$ then $-t \log |T'(x)| < \phi(x)$. If we denote by $P_{\epsilon} (\cdot)$ the pressure of 
$T$ restricted to the maximal $T-$invariant set in  $(0, \epsilon)$ we have that  $P_{\epsilon} (-(t+ \delta) \log |T'|) \leq P_{\epsilon} (q \phi - \delta \log|T'| )$. Since the entropy of $T$ restricted to $(0,1) \setminus  (0, \epsilon)$ is finite and the potentials $\phi$  and $\log |T'|$ restricted to this set are bounded, we can deduce that for $q>0$ and any positive $t>0$ we have
$$P(-(t+ \delta) \log |T'|) \leq P(q \phi - \delta \log|T'| ).$$
Therefore $\delta^*=s_{\infty}$.

This implies that if $s=s_{\infty}$ then $\delta(\alpha)=s$ for all $\alpha\in (-\infty,\alpha_M)$. So we will assume that $s_{\infty}<s$. If $\delta(\alpha)>s_{\infty}$ then by our assumption on $\phi$ we have $P(q(\phi-\alpha)-\delta(\alpha)\log |T'|)<\infty$ for all $q\in\R$ and 
$$\lim_{q\to\pm\infty}P(q(\phi-\alpha)-\delta(\alpha)\log |T'|)=\infty,$$
and so we must be in case 2 from Lemma \ref{options}. So we need to show that the set
$$J=\left \{\alpha:\delta(\alpha)>s_{\infty} \right\}$$
is a single interval. Denote by  $\nu$ be the equilibrium measure corresponding to  $-s\log |T'|$ and by $\underline{\alpha}=\int \phi\text{d}\nu$. Let $\alpha\in J$ we know that there is an equilibrium measure  $\mu_{\alpha}$, with $\int\phi\text{d}\mu_{\alpha}=\alpha$ and $\frac{h(\mu_{\alpha})}{\lambda(\mu_{\alpha})}=\delta(\alpha)$.  Let $\beta \in \R$ be real number bounded by $\alpha$ and $\underline{\alpha}$. By considering convex combinations of $\mu_{\alpha}$ and $\nu$ we can see that $\delta(\beta)>\delta(\alpha)$. It therefore follows that $J$ is a single  interval and the only possible points of non-analycity for $\delta(\alpha)=b(\alpha)$ are the endpoints of $J$.

\end{proof}

\section{The Lyapunov Spectrum}\label{sh}
A special case of the Birkhoff spectrum, which has received a great deal of attention, is the Lyapunov spectrum. This can be included in our setting by considering $\phi(x)=-\log |T'(x)|$ and then the Lyapunov spectrum is given by $L(\alpha)=b(-\alpha)$. The present section is devoted not only to show how  previous work on Lyapunov spectrum can be deduced from ours, but also to present new results on the subject.  

In related setting there has been work for the Gauss map in \cite{KS, pw}; for fairly general piecewise linear systems \cite{KMS} and in \cite{JK} the spectra for ratios of functions is studied where one of the functions is $-\log |T'(x)|$.

If $T$ denotes an expanding-Markov-Renyi (EMR) map then the variational formula proved in Theorem \ref{main} holds for the Lyapunov spectrum. On the other hand,  neither of the assumptions for Theorem  \ref{analytic} are satisfied. However, it is still possible to describe in great detail the Lyapunov spectrum. 

Let $\phi(x)=-\log |T'(x)|$ and, as in the previous Section,  let
$$s_{\infty}=\inf\{\delta \in \R:P(\delta\phi)<\infty\}.$$
The following Theorem shows how our results fit in with the results in \cite{KS, JK, KMS}. 

\begin{teo}\label{lyapunov}
For all $\alpha\in (-\infty,\alpha_M)$ we have that
\begin{equation}\label{Marc}
b(\alpha)=\inf_{u}\left\{u+\frac{P(u\phi)}{-\alpha}\right\}.
\end{equation}
Furthermore 
\begin{enumerate}
\item
If $P(s_{\infty}\phi)=\infty$ then $b(\alpha)$ is real analytic on $(-\infty,\alpha_M)$. 
\item
If $P(s_{\infty}\phi)=k<\infty$ and $\mu_c$ is the equilibrium state for $s_{\infty}\phi$ then $b(\alpha)$ is analytic except at $\alpha_c= \int \phi\text{d}\mu_c$.  For $\alpha\leq\alpha_c$ we have that $b(\alpha)=s_{\infty}-\frac{k}{\alpha}$.
\end{enumerate}
\end{teo}

\begin{proof}
The formula for $b(\alpha)$ given in equation \eqref{Marc} was shown in a slightly different setting in \cite{KMS}. We show how it can also be derived from the methods in this paper. For each $\delta \in \R$ we will denote the function $f_{\delta}:\R \rightarrow \R$ by
$$f_{\delta}(q)=P(q\phi+\delta\phi)= P((q +\delta) \phi).$$
We first assume that $P(s_{\infty}\phi)=\infty$. Thus, we have 
\begin{equation*}
f_{\delta}(q)=
\begin{cases}
\infty & \text{ if } q \leq s_{\infty} -\delta ;\\
\text{finite }  & \text{ if } q > s_{\infty} -\delta.
\end{cases}
\end{equation*}
Therefore, for each $\alpha\in (-\infty,\alpha_M)$ and for each $\delta \in \R$ there exist $q(\delta) > s_{\infty} -\delta$ such that $f'_{\delta}(q(\delta))=\alpha$. Denote by $q(\delta(\alpha)) \in \R$ the corresponding value for $\delta(\alpha)$. We have that
\[\frac{d}{dq} P(q \phi + \delta(\alpha) \phi) \Big|_{q=q(\delta(\alpha))}= \alpha.\] 
Moreover
\[P(q(\delta(\alpha)) \phi + \delta(\alpha) \phi)=q(\delta(\alpha)) \alpha.\]
Thus, for all $\alpha\in (-\infty, \alpha_M]$ we are in case \ref{case2} from Lemma \ref{options}.
%
% Note that the function
%\begin{equation*}
%\delta \to f_{\delta} (q(\delta))= P\left((q(\delta) + \delta) \phi\right),
%\end{equation*}
%is continuous (REAL ANALYTIC ACTUALLY). Furthermore, there exists an equilibrium  measure $\mu_0$ such that
%$$f_0(q(0))= P(q(0) \phi) = h(\mu_0) + q(0) \alpha \geq q(0) \alpha.$$
% Note also that if $s= \dim_H \Lambda$ there exists an invariant measure $\mu_s$ such that $f_s(q(s)) = h(\mu_s) +s \alpha + q(s) \alpha$. Since $P(s \phi)=0$, we have that $f_s(q(s)) \leq q(s) \alpha$. By the continuity of $ f_{\delta} (q(\delta))$ there exists $\delta_{\alpha}  \in [0, s]$ such that
% \[ f_{\delta_{\alpha}} (q(\delta_{\alpha}))=0.\]
%Note that 
%\[  q(\delta_{\alpha})+\delta_{\alpha}= \delta(\alpha)=: \sup \left\{ \frac{h(\mu)}{-\alpha} : \mu \in \mathcal{M}_T \text{ and }
%\int \log |T'|  d \mu= -\alpha  \right\}.\]
%Indeed, this follows just noticing that there exist an equilibrium measure $\mu_{\alpha}$ such that
%\[ P( (q(\delta_{\alpha}) + \delta_{\alpha}) \log |T'|) = h(\mu_{\alpha}) + (q(\delta_{\alpha}) + \delta_{\alpha}) \alpha =0.\]
%Therefore
%\[\frac{h(\mu_{\alpha})}{-\alpha}= q(\delta_{\alpha}) + \delta_{\alpha},\]
%any other invariant measure $\nu$ with $\int \log |T'| \ d \nu = \alpha$ must have entropy less than  $\mu_{\alpha}$.  Thus, for all $\alpha\in (-\infty, \alpha_M]$ we are in case \ref{case2} from Lemma \ref{options}. 
Therefore, the Lyapunov spectrum is real analytic  on $ (-\infty, \alpha_M]$.

%To deduce the formula \eqref{Marc} we can find $q_0$ where $f_{\delta(\alpha)}'(q_0)=\alpha$ and $f_{\delta(\alpha)}(q_0)=q_0\alpha$. 
Now let $\mathcal{P}:\R\rightarrow\R$ be defined by $\mathcal{P}(u):= u+\frac{P(u\phi)}{-\alpha}$. We can then deduce that $\mathcal{P}(q(\delta(\alpha))+\delta(\alpha))=\delta(\alpha)$ and $\mathcal{P}'(q(\delta(\alpha))+\delta(\alpha))=0$. Finally since the pressure is convex we must have that $\mathcal{P}''(q(\delta(\alpha)))>0$ and that $q(\delta(\alpha))$ will be the only minimum point for $\mathcal{P}$. Thus
$$\delta(\alpha)=\inf_{u}\left\{u+\frac{P(-u\log |T'|)}{\alpha}\right\}.$$

We will now assume that  $P(-s_{\infty}\phi)=k$. Let $\mu_c$ be the equilibrium measure associated to $-s_{\infty}\phi$. If $\alpha\geq\int\phi\text{d}\mu_c$ then we can argue exactly as in the previous case. 
For $\alpha<\alpha_c$ we let $q=\frac{k}{\alpha}$ and note that
$$P(q(\phi-\alpha)+(s_\infty-q)\phi)=P(s_{\infty}\phi)-\alpha q=k-k=0.$$
Note that if $q < k/\alpha$ then 
\[ P\left(q\left(\phi-\alpha\right)+\left(s_\infty-\frac{k}{\alpha}\right)\phi\right)  = \infty.\]
Therefore we are in case \eqref{case3} from Lemma \ref{options}. Thus $b(\alpha)=s_{\infty}-\frac{K}{\alpha}$. We again define $\mathcal{P}:\R\rightarrow\R$ by $\mathcal{P}(u):= u+\frac{P(-u\phi)}{\alpha}$. If $q \in \R$ is such that $\mathcal{P}(q)<\infty$ then  we denote by $\mu_q$ be the equilibrium measure associated with $q\phi$.   Note that $\mathcal{P}'(q)=1+{-\lambda(\mu_q)}{-\alpha}>0$. Thus the infimum of $\mathcal{P}$ will be achieved at $s_{\infty}$. We can then calculate. 
$$\inf_{u}\{\mathcal{P}(u)\}=\mathcal{P}(s_{\infty})=s_{\infty}-\frac{k}{\alpha}.$$
\end{proof}

We now turn our attention to   the shapes the Lyapunov spectrum can take. We start by giving a result which holds for all potentials $\phi\in\mathcal{R}$  
\begin{teo}\label{shape}
Let $T$ be an EMR map with $\dim_H \Lambda=s$ and  $\phi\in\mathcal{R}$, then
\begin{enumerate}
\item
If there exists a $T$-ergodic measure of maximal dimension $\mu$  and $\alpha^*=\int\phi{\text{d}\mu}$ then $b(\alpha)$ is non-increasing on $[\alpha^*,\alpha_M]$ and non-decreasing on $(-\infty,\alpha^*]$. (Note that it is possible that $\alpha^*=-\infty$.)
\item
If there exists no ergodic measure of maximal dimension then $b(\alpha)$ is non-decreasing on $(-\infty,\alpha_M]$.  
\end{enumerate}
\end{teo}
\begin{proof}
For the first part. Let $\alpha_1>\alpha_2>\alpha^*>-\infty$. For any $\epsilon>0$ there exists an invariant measure $\mu_1$ such that $\int\phi\text{d}\mu_1=\alpha_1$ and $h(\mu_1)\geq\lambda(\mu_1)(b(\alpha_1)-\epsilon)$. If $\alpha^*>-\infty$ we can then find $p\in (0,1)$ such that $\alpha_2=p\alpha^*+(1-p)\alpha_1$. Now let $\nu_1=p\mu+(1-p)\mu_1$. Thus $\int\phi\text{d}\nu_1=\alpha_2$ and 
$$h(\nu_1)\geq ps\lambda(\mu_1)+(1-p)(b(\alpha_1)-\epsilon)\lambda(\mu_1)\geq (b(\alpha_1)-\epsilon)\lambda(\nu_1).$$
Therefore $b(\alpha_2)\geq b(\alpha_1)$. The case where $\alpha_1<\alpha_2<\alpha^*$ is handed analogously. Now assume that $\alpha^*=\infty$ and $\alpha_1>\alpha_2$. Let $\alpha_M>\alpha_1>\alpha_2>-\infty$. By considering compact approximations we can find an invariant measure $\mu$ such that $\int\phi\text{d}\mu<\alpha_2$ and $\infty>h(\mu)\geq (b(\alpha_1)-\epsilon)\lambda(\mu)$. We can also find a measure $\mu_1$ such that $\int\phi\text{d}\mu_1<\alpha_1$ and $h(\mu_1)\geq (b(\alpha_1)-\epsilon)\lambda(\mu_1)$. To complete the proof we take a suitable convex combination of $\mu$ and $\mu_1$. 

In the case where there is no ergodic measure of maximal dimension we know that $s=s_{\infty}$. Again by considering compact approximations we can find a sequence of invariant measures $\mu_n$ such that $\lim_{n\rightarrow\infty}\phi\text{d}\mu_n=-\infty$ and $\lim_{n\rightarrow\infty}\frac{h(\mu_n)}{\lambda(\mu_n)}=s$. The proof now simply follows the first part when $\alpha^*=-\infty$.  
\end{proof}
We now return to the Lyapunov spectrum. It was shown in \cite{ik} that  in the hyperbolic case it can have inflection points and it clearly has to have such points in the non-compact case. An application of the methods used in Theorem \ref{shape} combined with results from Theorem \ref{lyapunov} allow us to prove in a simple way that as long as $s_{\infty}<s=\dim_H(\Lambda)$ the inflection points can only appear in the decreasing part of the spectrum. We present the proof in the non-compact case however it also holds in the compact, hyperbolic case.

\begin{coro}
Let $T$ be an EMR map such that  $s_{\infty}<s=\dim_H(\Lambda)$ then  the increasing part of the Lyapunov spectrum is concave.
\end{coro}
\begin{proof}
Again we will let $\phi=-\log |T'|$ and note that in this case the Lyapunov spectrum satisfies  $L(\alpha)=b(-\alpha)$. 
Since $s_{\infty}<s$ there exists an ergodic measure of maximal dimension that we denote by $\mu$. Let $\int\phi\text{d}\mu=\alpha^*$. By Theorem \ref{shape} we know that $b(\alpha)$ is non-increasing on $[\alpha^*,\alpha_M)$. Moreover the proof of Theorem \ref{lyapunov} implies that for all $\alpha\in [\alpha^*,\alpha_M)$ there will exist a measure $\mu_{\alpha}$ such that $\lambda(\mu_{\alpha})=-\alpha$ and $\frac{h(\mu_{\alpha})}{\lambda(\mu_{\alpha})}=\delta(\alpha).$    

 We now introduce variables $\lambda_1,\lambda_2$ such that 
$$\inf \left\{ \int \log |T'| \ d \nu : \nu \in \mathcal{M}_T \right\}:= \lambda_m <\lambda_1 <\lambda_2 < \lambda^*:= \int \log |T'| \ d \mu.$$
 Thus we  we can find $\mu_1, \mu_2 \in \mathcal{M}_T$ such that $L(\lambda_1)=\dim_H \mu_1$, $L(\lambda_2)=\dim_H \mu_2$, $\lambda(\mu_1)=\lambda_1$ and $\lambda(\mu_2)=\lambda_2$.  
Let
 \[L(t):= \frac{th(\mu_2) +(1-t)h(\mu_1)}{t\lambda(\mu_2)+(1-t)\lambda(\mu_1)} \]
 for $t\in [0,1]$.
In order to study the convexity properties of the Lyapunov spectrum $L(\alpha)$ we compute the derivatives of the function $L(t)$ and note that $L(t\lambda_1+(1-t)\lambda_2)\geq L(t)$ with equality when $t=0,1$. The derivative of $L(t)$ is,
 \begin{equation} \label{der} 
L'(t)= \frac{h(\mu_2) \lambda(\mu_1)- h(\mu_1) \lambda(\mu_2)}{(t \lambda(\mu_2) +(1-t) \lambda(\mu_1))^2  }.
\end{equation}
The second derivative is given by:
\begin{equation} \label{2der}
L''(t)= \frac{2(h(\mu_2) \lambda(\mu_1)- h(\mu_1) \lambda(\mu_2))}{(t \lambda(\mu_2) +(1-t) \lambda(\mu_1))^3} \left( \lambda(\mu_1) -\lambda(\mu_2)  \right)
\end{equation}
Note that all the Lyapunov exponents are positive therefore the denominator of \eqref{2der} is positive. Since
\[  \frac{h(\mu_1)}{\lambda(\mu_1)} = \dim_HJ(\lambda_1) < \dim_H J(\lambda_2) = \frac{h(\mu_2)}{\lambda(\mu_2)}, \]
we have that $2(h(\mu_2) \lambda(\mu_1)- h(\mu_1) \lambda(\mu_2))>0$. Therefore the sign of \eqref{2der} is determined by the sign of $ \lambda(\mu_2) -\lambda(\mu_1) $.
Which by definition satisfies $\lambda_1= \lambda(\mu_1) <\lambda(\mu_2)=\lambda_2$. Therefore  $L''(t) <0$ and the function $L(\alpha)$ is concave on $[\lambda_m, \lambda^*]$. 
\end{proof}

In the case where $s=s_{\infty}$ then if $P(s_{\infty}\phi)=\infty$ then the above proof can be easily adapted to show the Lyapunov spectrum is concave. 
%
%If $P(s_{\infty}\phi)=P(s\phi)=-K<0<\infty$ and $\mu_c$ is the equilibrium state for $s\phi$ then there is a possibility of a point where the spectrum is non-concave at $\lambda(\mu_c)$. 

%Since we were not able to find in the literature any theorem describing the shape of the multifractal spectrum of Birkhoff averages for expanding maps of the interval with finitely many branches, we provide here the results we require in order to prove Theorem \ref{shape}. 

%\begin{rem}
%The above argument shows that if $\phi=\log |F'|$, that is if we consider the Lyaunov spec..
%\end{rem}

\section{Examples} \label{ejem}
An irrational number $ x \in (0,1)$  can be written as a continued fraction of the form
\begin{equation*}
x = \textrm{ } \cfrac{1}{a_1 + \cfrac{1}{a_2 + \cfrac{1}{a_3 + \dots}}} = \textrm{ } [a_1 a_2 a_3 \dots],
\end{equation*}
where $a_i \in \mathbb{N}$. For a general account on continued fractions see \cite{hw, k}. %The $n-th$ approximant
%$p_n(x) / q_n(x)$ of the number $x \in [0,1]$ is defined by
%\begin{equation} \label{approx}
%\frac{p_n(x)}{q_n(x)} = \cfrac{1}{a_1 + \cfrac{1}{a_2 + \cfrac{1}{\dots + \frac{1}{a_n}}}} 
%\end{equation}
The Gauss map (see Example \ref{ga})  $G :(0,1] \to (0,1]$, is the interval map defined by 
\[G(x)= \frac{1}{x} -\left[ \frac{1}{x} \right]. \]
This map is closely related to the continued fraction expansion. 
Indeed, for $0 < x <1$ with $x=[a_1 a_2 a_3 \dots ]$ we have that
$a_1=[1/x], a_2=[1/Gx], \dots, a_n=[1/G^{n-1}x]$. In particular, the Gauss
map acts as the shift map on the continued fraction expansion,
\begin{equation*}
 a_n = \Big[1/G^{n-1}x \Big].
\end{equation*} 
The following result was initially proved by Khinchin \cite[p.86]{k} in the case where $\phi(n)<Cn^{1/2-\rho}$.
\begin{teo}[Khinchin] \label{kin}
Let $\phi :\N \to \R$ be a non-negative potential. If there exists constants
$C>0$ and $\rho >0$ such that for every $n \in \N$, 
\[  \phi(n) <  C n^{1 -\rho},              \]
then for Lebesgue almost every $x \in (0,1)$ we have that
\begin{equation*}
 \lim_{n \to \infty} \frac{1}{n} \sum_{i=0}^{n-1} \phi(G^i x) = \sum_{n=1}^{\infty}
 \left( \phi(n) \frac{\log \left( 1 + \frac{1}{n(n+2)}          \right)}{\log 2}             \right).
 \end{equation*}
\end{teo}

\begin{rem}
The above results directly follows form the ergodic theorem applied to the (locally constant) potential $\phi$ with respect to the (ergodic) Gauss measure,
\[\mu_G(A)= \frac{1}{\log 2}\int_A \frac{dx}{1+x}.\] 
The Gauss measure is absolutely continuous with respect to the Lebesgue measure. Moreover,  it is the measure of maximal dimension for the map $G$. 
\end{rem}

As a direct consequence of  Theorem \ref{main} we  can compute the Hausdorff dimension of the  level sets determined by the potential $\phi$ (strictly speaking we should apply our results to the potential $-\phi$, but clearly this does not make any difference).
Indeed, first note that  potentials satisfying the assumptions of Khinchin's Theorem such that $\lim_{n \to \infty} \phi(n)= \infty$ satisfy the assumptions of Theorem \ref{main}. That is, if  $\phi :(0,1) \to \R$ is a non-negative potential such that
\begin{enumerate}
\item if $x \in (0,1)$ and $x=[a_1, a_2 \dots]$ then $\phi(x)=\phi(a_1)$,
\item there exists constants $C>0$ and $\rho >0$ such that for every $n \in \N$ and $x \in (1/(n+1), 1/n)$ , 
\[  \phi(x)=\phi(n) <  C n^{1 -\rho},  \]
\item $\lim_{x \to 0} \phi(x)  = \infty$,        
\end{enumerate}
then $\phi \in \mathcal{R}$. Our first result in this setting is the following immediate Corollary to Theorem \ref{main}.
\begin{coro} \label{kinch}
Let $\phi \in \mathcal{R}$.
Then if we denote by
\[K(\alpha) := \left\{ x \in (0,1) :   \lim_{n \to \infty} \frac{1}{n} \sum_{i=0}^{n-1} \phi(G^ix) =\alpha  \right\}, \]
we have that
\begin{equation}
\dim_H(K(\alpha))=  \sup \left\{ \frac{h(\mu)}{\lambda(\mu)} :\mu \in \mathcal{M}_G, \int  \phi \, d \mu = \exp(\alpha) \textrm{ and } \lambda(\mu) < \infty \right\}.
\end{equation}
\end{coro}

A particular case of the above Theorem has received a great deal  of attention. If $\phi(x) = \log a_1$ then the Birkhoff average can be written as the so called \emph{Khinchin function}:
\[ k(x):=    \lim_{n \to \infty} \left( \log \sqrt[n]{a_1 \cdot a_2 \cdot  \ldots  \cdot a_n} \right).\]
This was first studied by Khinchin who proved that

\begin{prop}[Khinchin]
 Lebesgue almost every number is such that
\begin{equation*}
\lim_{n \to \infty} \left( \log \sqrt[n]{a_1 \cdot a_2 \cdot  \ldots  \cdot a_n} \right) =\log \left( \prod_{n=1}^{\infty}
\left( 1 + \frac{1}{n(n+2)}  \right)^{\frac{\log n}{\log 2}} \right)= 2.6 ...
\end{equation*}
\end{prop}

Recently, Fan et al \cite{fl}  computed the Hausdorff dimension of the  level sets
determined by the Khinchin function. They obtained the following result
 $\int \log a_1 \, d \mu_G := \alpha_{SRB} < \infty$,
\begin{prop}\label{geom}
The function 
\[b(\alpha):= \dim_H \left( \left\{  x \in (0,1) :   \lim_{n \to \infty}  \log \left( \sqrt[n]{a_1 \cdot a_2 \cdot  \ldots  \cdot a_n} \right) = \alpha   \right\}  \right), \]  
is real analytic, it is strictly increasing
and strictly concave in the interval  $[\alpha_m, \alpha_{SRB})$ and it is decreasing  and has an inflection point in $( \alpha_{SRB}, \infty).$
\end{prop}

An interesting family of related examples is given by letting $\gamma>0$ and considering the locally constant potential
$\phi_{\gamma}([a_1,a_2, \ldots])=-a_1^{\gamma}$. For this potential the Birkhoff average is given by
\begin{equation} \label{arithmetic}
\lim_{n \to \infty} \frac{1}{n} \sum_{i=0}^{n-1} \phi_{\gamma}(G^ix)= -\lim_{n \to \infty} \frac{1}{n} \left(a_1^{\gamma}+a_2^{\gamma} +\dots +a_n^{\gamma}\right),
\end{equation}
where $x=[a_1,a_2, \dots , a_n, \dots]$. 
Let us note that if $\gamma\geq 1$ then for Lebesgue almost every point $x \in (0,1)$ the limit defined in \eqref{arithmetic} is not finite. For $\gamma<1$ we let $G(\gamma):=\int\phi_{\gamma} \text{d}\mu_G>-\infty$.   Nevertheless for any $\gamma>0$ we have that $\phi_{\gamma} \in \mathcal{R}$, so the following result is a direct corollary of Theorem \ref{main}
\begin{coro} \label{ari}
Denote by
\[A(\alpha,\gamma) := \left\{ x \in (0,1) :   \lim_{n \to \infty} \frac{1}{n} \left(a_1^{\gamma}+a_2^{\gamma} +\dots +a_n^{\gamma}\right) =\alpha  \right\}, \]
we have that
\begin{equation}
\dim_H(A(\alpha,\gamma))=  \sup \left\{ \frac{h(\mu)}{\lambda(\mu)} :\mu \in \mathcal{M}_G, \int  A \, d \mu = -\alpha \textrm{ and } \lambda(\mu) < \infty \right\}.
\end{equation}
\end{coro}
We can also use Theorem \ref{analytic} to give more detail about the function $\alpha  \to \dim_H(A(\alpha, \gamma))$. 
\begin{prop}
Let $\gamma >0$ then 
\begin{enumerate}
\item
If $\gamma\geq 1$ the function $\alpha  \to \dim_H(A(\alpha, \gamma))$ is real analytic and it is strictly increasing. 
\item
If $0 < \gamma<1$ the function $\alpha  \to \dim_H(A(\alpha,\gamma))$ is real analytic on $[G(\gamma),\alpha_M)$ and for $\alpha<G(\gamma)$ we have  $\dim_H(A(\alpha,\gamma))=1$.  
\end{enumerate}
\end{prop}
\begin{proof}
Since $\lim_{x\to 0}\frac{\phi_\gamma(x)}{-\log |T'(x)|}=\infty$ the Theorem immediately follows from the first part of Theorem \ref{analytic}.
\end{proof}
The sets $A(\alpha, 1)$ are related to the sets where the frequency of digits in the continued fraction is prescribed. The Hausdorff dimension of these sets was recently computed in \cite{flm}.

We conclude this section exhibiting explicit examples of dynamical systems and potentials for which the behaviour of the Birkhoff spectra is complicated.

A version of  following example appears in \cite{sa2}. Consider the partition of the interval $[0,1]$ given by the sequence of points of the form $x_n= 1 / (n (\log 2n)^2)$ together with the points  $\{0, 1\}$.  Let $T$ be the EMR map  defined 
on each of the intervals generated by this sequences to be linear, of positive slope and onto. Then
\begin{equation*}
P(-t \log |T'|)=
\begin{cases}
\infty & t < 1 \\
\text{finite}  & t \geq 1.
\end{cases}
\end{equation*}
Moreover,  $P(-\log|T'|)=1$ and for $t >1$ we have $P(-t \log |T'|) <0$. Therefore,  $\dim_H \Lambda=s=s_{\infty}=1$. Choose now $\phi\in\mathcal{R}$ such that $\lim_{x\to 0}\frac{\phi(x)}{-\log |T'(x)|}=0$. We can then see that for any $\delta<1$ we have  $P(q\phi-\delta\log |T'|)=\infty$ for all $q\in\R$ and so $\delta^*=1$. Therefore, it is a consequence of Lemma \ref{lb} that $b(\alpha)=1$ for all $\alpha\in (-\infty,\alpha_M]$.  Other examples of dynamical systems satisfying these assumptions can be found in \cite{mu}.

\section{Hausdorff dimension of the extreme level sets} \label{extr}
This section is devoted to study the Hausdorff dimension of one of the two extreme level sets. Since the potentials we have considered are not bounded the level set
\[J(-\infty):= \left\{  x \in (0,1): \lim_{n \to  \infty} \frac{1}{n} \sum_{i=0}^{n-1} \phi(T^i x) = -\infty          \right\} \]
can have positive Hausdorff dimension. In this section we compute it. 
\begin{teo}\label{extreme}
Let $\phi \in \mathcal{R}$ then
\begin{equation}\label{extremeq}
\dim_H(J(-\infty))=\lim_{\alpha\rightarrow-\infty} F(\alpha).
\end{equation}
\end{teo}
\subsection*{Proof of Theorem \ref{extreme}}
To start we need a lemma showing that the limit on the right hand side of equation (\ref{extremeq}) does indeed exist.
\begin{lema}\label{existence}
There exists $s\in [0,1]$ such that $\lim_{\alpha\rightarrow-\infty} F(\alpha)=s$.
\end{lema}
\begin{proof}
The limit clearly exists since by Theorem \ref{shape} the function $\alpha\to F(\alpha)$ is monotone when $-\alpha$ is sufficiently large.
\end{proof}
In order to prove the upper bound, 
$$\dim_H(J(-\infty))\leq\lim_{\alpha\rightarrow-\infty} F(\alpha),$$
we first give a uniform lower bound for $\lim_{\alpha\rightarrow-\infty} F(\alpha)$.
\begin{prop}\label{lowb}
Let $t^*$ be the critical value for the pressure of the potential $-\log |T'|$. We have that
$\lim_{\alpha\rightarrow-\infty} F(\alpha) \geq t^*$.
\end{prop}
\begin{proof}
Consider the sets $\Lambda_n=\pi(\{n,n+1,\ldots\}^\N)$. Note that  
$\dim_H \Lambda_n\geq t^*$ by the definition of $t^*$. However, for any $\epsilon>0$ the set $\Lambda_n$ will support a $T$-invariant measure $\mu_n$ with $\lambda(\mu_n)<\infty$, $\frac{h(\mu_n)}{\lambda(\mu_n)}\geq\dim_H \Lambda_n-\epsilon$ and $\int\phi \ d\mu_n>-\infty$. We also have that $\lim_{n\rightarrow\infty}\int \phi\ d \mu_n=-\infty$. The result now follows.
\end{proof}
We now fix $\alpha\in\R$ and consider the set
$$J(\alpha,N)=\left\{x\in\Lambda:\frac{S_k\phi(x)}{k}\leq\alpha\text{, for every }k\geq N\right\}.$$
It is clear that $J(-\infty)\subset\cup_{N\in\N}J(\alpha,N)$. Thus it suffices to show that for all $N\in\N$ 
$$\dim_H J(\alpha,N)\leq\sup_{\beta>\alpha} F(\beta).$$

Fix $N\in\N$ and for $k\in\N$ let
$$C_k(\alpha)=\{I(i_1,\ldots,i_k):I(i_1,\ldots,i_k)\cap J(\alpha,N)\neq\emptyset\}.$$
Let $\epsilon>0$ and note that if for infinitely many $k$ we have
$$\sum_{I(i_1,\ldots,i_k)\in C_k(\alpha)}|I(i_1,\ldots,i_k)|^{t^*+\epsilon}\leq 1$$
then $\dim_H J(\alpha,N)\leq t^*+\epsilon\leq\lim_{\alpha\rightarrow-\infty} F(\alpha)+\epsilon$. So we may assume that there exists $K\in\N$ such that for $k\geq K$
$$1<\sum_{I(i_1,\ldots,i_k)\in C_k(\alpha)}|I(i_1,\ldots,i_k)|^{t^*+\epsilon}<\infty.$$
Note that the sum must be convergent because $t^*+\epsilon$ is greater than the critical value $t^*$.
Thus for each $k\geq K$ we can find $t_k$ such that
$$\sum_{I(i_1,\ldots,i_k)\in C_k(\alpha)}|I(i_1,\ldots,i_k)|^{t_k}=1.$$
It follows that $\dim_H J(\alpha,N)\leq\limsup_{k\rightarrow\infty}t_k$. To complete the proof we need to relate $t_k$ to the entropy and Lyapunov exponent of an appropriate $T$-invariant measure.

Since $C_k(\alpha)$ contains infinitely many cylinders we need to consider a finite subset of $C_k(\alpha)$, that we denote by $D_k(\alpha)$, where
$$\sum_{I(i_1,\ldots,i_k)\in D_k(\alpha)}|I(i_1,\ldots,i_k)|^{t_k}=A\geq 1-\epsilon.$$
As in the proof of Lemma \ref{cover} we let $\eta_k$ be the $T^k$ invariant measure which assign each cylinder in $D_k(\alpha)$ the measure $\frac{1}{A}|I(i_1,\ldots,i_k)|^{t_k}$. Note that there will exist $C>0$ such that for all $k\geq K$ the Lyapunov exponent $\lambda(\eta_k, T^{k+1})$ satisfies 
$$\left|-\lambda(\eta_k,T_k)-\frac{1}{A}\sum_{I(i_1,\ldots,i_k)\in D_k(\alpha)} |I(i_1,\ldots,i_k)|^{t_k}\log |I(i_1,\ldots,i_k)|\right|\leq C.$$
Computing the entropy with respect to $T^k$ of $\eta_k$ gives
$$h(\eta_k,T^k)=\sum_{I(i_1,\ldots,i_k)\in D_k(\alpha)} \frac{t_k}{A}|I(i_1,\ldots,i_k)|^{t_k}\log |I(i_1,\ldots,i_k)|+\log A.$$
Since $A\geq 1-\epsilon$ and $\lambda(\eta_k,T^k)\geq\xi^{k}$ it follows that $\lim_{k\rightarrow\infty}\frac{h(\eta_k,T^k)}{\lambda(\eta_k,T^k)}-t_k=0$.
Since $\eta_k$ is compactly supported we know that $\int\phi\ d \eta_k>-\infty$ and by the distortion property $\limsup_{k\rightarrow\infty}\int\phi\ d \eta_k\leq\alpha$. To finish the proof we simply let $\mu_k=\sum_{i=0}^{k-1}\eta_k\circ T^{-i}$.

To prove the lower bound we use the method of constructing a w-measure as done by Gelfert and Rams in \cite{GR}. We will let $\lim_{\alpha\rightarrow-\infty} F(\alpha)=s$ and start by observing that
there exists a sequence of ergodic measures $\{\mu_n\}_{n\in\N}$ where $\lim_{n\rightarrow\infty}\int\phi\ d \mu_n=-\infty$, $\lambda(\mu_n)<\infty$, for all $n\in\N$, $\frac{\lambda(\mu_{n+1})}{\lambda(\mu_n)}\leq 2$ and $\lim_{n\rightarrow\infty}\frac{h(\mu_n)}{\lambda(\mu_n)}=s$.
We now let $\epsilon>0$ and assume that for all $n$ $\frac{h(\mu_n)}{\lambda(\mu_n)}\geq s-\epsilon$. For all $i\in\N$ by Egorov's Theorem we can find $\delta>0$ and $n_i\in\N$ such that there exists a set $X_i(\delta)$ where for all $n\geq n_i$ and $x\in X_i(\delta)$
\begin{enumerate}
\item
$S_n\phi(x)\leq n(\alpha_i+\epsilon)$.
\item
$(s+\epsilon)(-\log|C_n(x)|)\leq-\log\mu_i(C_n(x))\geq (s-\epsilon)(-\log|C_n(x)|)$
\item
$-\log|C_n(X)|\in (n(\lambda(\mu_i)-\epsilon),n(\lambda(\mu_i)+\epsilon))$.
\item
$\mu_i(X_i(\delta))\geq 1-\delta$.

\end{enumerate}  
 We can let $k_1=n_1+\left[\frac{n_2}{\delta}\right]+1$ and $k_{i}=\left[\frac{\left(\sum_{l=1}^{i-1}k_{l}\right)+\lambda(\mu_{i+1})n_{i+i}}{\delta}\right]+1$. We let $Y_i$ be all $k_i$ level cylinders with nonzero intersection with $X_i(\delta)$. We then define $Y$ to be the space such that $x\in Y$ if and only if $T^{\sum_{l=1}^{j-1} k_l}(x)\in Y_j$ for all $j\in\N$. We will need to consider the size of $n-$th level cylinders for points in $Y$. We get the following lemma
\begin{lema}\label{cylinder}
 There exists $K(\epsilon)>0$ such that $\lim_{\epsilon\to 0}K(\epsilon)=0$  
and for all $x\in Y$ and $n$ sufficiently large
$$\nu(B(x,|C_n(x)|))\leq (1+K(\epsilon))^n\nu(|C_{n+1}(x)|).$$
\end{lema} 
\begin{proof}
To proof this we use the condition in the definition of $Y$. For any $x,y\in Y$ we need to compare the diameter of $C_n(x)$ and $C_{n+1}(y)$ and the measure of $C_{n+1}(y)$ and $C_{n+1}(x)$ . We consider the case when $k_i\leq n\leq k_i+n_i-1$ we then have that for all $x,y\in Y$
$$-\log |C_{n+1}(y)|\leq \sum_{j=1}^{i} k_j(\lambda(\mu_j)+\epsilon)+n_{i+1}(\lambda(\mu_{i+1})+\epsilon)+\sum_{j=2}^{n+1}\var_k (\log |T'|)$$
and
$$-\log |C_{n}(x)|\geq \sum_{j=1}^{i} k_j(\lambda(\mu_j)-\epsilon)-\sum_{j=2}^{n+1}\var_k(\log |T'|).$$
In the case where $k_i+n_i\leq n\leq k_{i+1}$ we simply have that for all $x,y\in Y$
$$|\log |C_{n+1}(y)|-\log |C_n(x)||\leq n\epsilon+\sum_{j=2}^{n+1}\var_k(\log |T'|)+\lambda_{i+1}.$$
We can thus deduce that there exists $Z(\epsilon)$ such that $\lim_{\epsilon\to 0}Z(\epsilon)=0$ and
$$\nu(B(x,|C_n(x)|))\leq (1+Z(\epsilon_n))\max_{y\in Y} \nu(B(y,|C_{n+1}(y)|)).$$
To complete the proof we need a uniform estimate of $\frac{\nu(B(y,|C_{n+1}(y)|))}{\nu(B(x,|C_{n+1}(y)|))}$ for all $x,y\in Y$. This follows from the definition of $Y$. 
\end{proof} 
  We can then define a measure supported on $Y$ as follows. Let $\nu_i$ be the measure which gives each cylinder in $Y_i$ equal weight. We then take the measures
$$\otimes_{j=1}^{l}\sigma^{\sum_{m=1}^{j-1} k_m}\nu_j$$
and note that this can be extended to a measure $\nu$ supported on $Y$.
% Kolmogorov Extension Theorem.
\begin{lema}
For all $x\in Y$ we have that  $\lim_{n\rightarrow\infty}\frac{S_n\phi(x)}{n}=-\infty$ and 
$$\dim_H Y\geq\dim_H \nu\geq s-C(\delta),$$
for some constant $C(\delta)>0$ where $C(\delta)\rightarrow 0$ as $\delta\rightarrow 0$. 
\end{lema}
\begin{proof}
For convenience we will let $z_i=\sum_{l=1}^i k_i$. By our definition of $k_i$ we will have that $\frac{n_{i+1}}{z_i}\leq\delta$.
If $x\in Y$ then we have that for $n\in \left[z_i,z_i+n_{i+1}\right]$, 
$$S_n\phi(x)\leq (\alpha_i+\epsilon)k_i+(\max_{x\in\Lambda}\{\phi(x)\})(n-z_i)+\sum_{j=1}^{\infty} V_j(\phi).$$
 Moreover for $n\in \left[z_i+n_{i+1},z_{i+1}\right]$ we have that 
$$S_n\phi(x)\leq (\alpha_i+\delta)z_i+(n-z_i)(\alpha_{i+1}+\delta)+\sum_{j=1}^\infty V_j(\phi).$$
Combining these two estimates and the definition of $k_i$ we obtain that $\lim_{n\rightarrow\infty} \frac{S_n\phi(x)}{n}=-\infty$.
To find a lower bound for $\dim\nu$ we need to find a lower bound for $\lim_{r\to 0}\frac{\log(\nu(B(x,r))}{\log r}$ for all $x\in Y$.  To start we let $x\in Y$, $n\in \left[z_i,z_i+n_{i+1}\right]$ and note that by the definition of $k_i$ this will mean that
$$\frac{\log C_{Z_i}(x)}{C_{n}(x)}\geq (1-\delta).$$
By the definition of $\nu$ we have that
$$\log\nu(C_n(x))\leq -i\log\delta+(s-\epsilon)\sum_{l=1}^i\log \left|C_i\left(\left(T^{Z_l}(x)\right)\right)\right|$$
which then gives using distortion estimates that
$$\log\nu(C_n(x))\leq -i\log\delta+(s-\epsilon)\log|C_{z_i}(x)|+\sum_{j=1}^{\infty} V_j(\log |T'|).$$
For $n\in \left[\sum_{l=1}^ik_l+n_{i+1},\sum_{l=1}^{i+1}k_l\right]$ we have that
$$\log\nu(C_n(x))\leq -i\log\delta+(s-\epsilon)\left(\left(\sum_{l=1}^i\left(\log \left|C_i(T^{z_l}(x))\right|\right)\right)+|C_{n-z_i}|(T^{z_i}(x))\right).$$
Again by applying distortion estimates we get that
$$\log\nu(C_n(x))\leq -i\log\delta+(s-\epsilon)\log|C_n(x)|+\sum_{j=1}^{\infty} V_j(\log |T'|).$$
Thus for all $n\in [z_i,z_{i+1}]$ we have that
$$\frac{\log\nu(C_n(x))}{\log |C_n(x)|}\geq \frac{-i\log\delta}{\log|C_n(x)|}+(1-\delta)(s-\epsilon)+\frac{\sum_{j=1}^{\infty} V_j}{\log|C_n(x)|}$$
and taking the limit as $i\to\infty$ gives that
$$\frac{\log\nu(C_n(x))}{\log |C_n(x)|}\geq (1-\delta)(s-\epsilon).$$

Now fix $r>0$ and $n$ such that $C_n(x)\geq r>C_{n+1}(x)$. We then have by Lemma  that for $n$ sufficiently large
\begin{eqnarray*}
\nu(B(x,r)&\leq& \nu(B(x,|C_n(x)|))\\
&\leq& (1+k(\epsilon))^n\nu(|C_{n+1}(x)|)\\  
&\leq&  (1+k(\epsilon))^n |C_{n+1}(x)|^{s-\epsilon}\leq (1+k(\epsilon))^n r^{s-\epsilon}.
\end{eqnarray*}
The proof is obtained by noting that $\frac{\log(1+K(\epsilon))^n}{\log r}$ can be made arbitrarily small by choosing $\epsilon$ sufficiently small.
\end{proof} 
The proof of Theorem \ref{extreme} is now finished. We finish this section by noting that combining Theorem \ref{extreme} and Proposition \ref{lowb} gives that for all $\phi\in\mathcal{R}$, $\dim J(-\infty)\geq t^*$.
\

\end{document}